\documentclass{amsart}
\usepackage[utf8]{inputenc}
\usepackage[utf8]{inputenc}
\usepackage{amsfonts}
\usepackage{amsmath}
\usepackage{amssymb}
\usepackage{mathrsfs}
\usepackage{amsthm}
\usepackage{enumerate}
\usepackage{float}
\usepackage{enumitem}
\usepackage{mathabx}
\usepackage{color}
\usepackage{caption}
\usepackage{subcaption}
\usepackage{thmtools}
\usepackage{enumitem}
\usepackage{thm-restate}
\usepackage{comment}
\usepackage{cleveref}

\newcommand{\dd}[1]{\mathrm{d}{#1}}
\DeclareMathOperator{\R}{{\mathbb{R}}}

\DeclareMathOperator{\I}{\mathscr{I}}

\DeclareMathOperator{\J}{\mathcal{J}(\delta)}
\DeclareMathOperator{\Ir}{\mathscr{I}_\mathrm{rem}(\delta)}

\DeclareMathOperator{\dl}{\delta}
\newcommand{\card}[1]{\mathrm{card}(#1)}
\newtheorem{theorem}{Theorem}[section]
\newtheorem{remark}[theorem]{Remark}
\newtheorem{lemma}[theorem]{Lemma}

\newtheorem{prop}[theorem]{Proposition}
\newtheorem{corollary}[theorem]{Corollary}
\newtheorem{claim}{Claim}[section]
\newtheorem{definition}[theorem]{Definition}

\usepackage{graphicx}

\DeclareMathOperator{\dist}{dist}
\DeclareMathOperator{\lamp}{\Lambda({\mathit p})}
\title[Bochner--Riesz operators with rough boundary]{Improved $L^p$ bounds for Bochner--Riesz Operators associated with rough convex domains in the plane}
\author{Hrit Roy }
\date{September 2023}

\begin{document}

\maketitle
\begin{abstract}
    We consider generalized Bochner--Riesz operators associated with planar convex domains. We construct domains which yield improved $L^p$ bounds for the Bochner--Riesz over previous results of Seeger--Ziesler and Cladek. The constructions are based on the existence of large $B_m$ sets due to Bose--Chowla, and large $\Lambda(p)$ sets due to Bourgain. 
\end{abstract}
\section{Introduction}\label{sec: main result}
\subsection{Main results}
Let $\Omega\subset\R^2$ be an open, bounded, convex set that contains the origin in its interior. We call such a set a \textit{convex domain}. Let $\rho$ be the associated Minkowski functional, which is defined as \begin{equation}\rho(\xi):=\inf\{t>0:\xi\in t\Omega\},\qquad\xi\in\R^2.\end{equation}
The Bochner--Riesz operator of order $\alpha>0$ associated with $\Omega$ is given by\footnote{Here $x_+:=\max\{x,0\}$.}$$(B^\alpha_\Omega f)\;\widehat{}\;(\xi):=(1-\rho(\xi))_+^\alpha\widehat{f}(\xi).$$ Let $\kappa_\Omega\in[0,\frac{1}{2}]$ be the affine dimension of $\Omega$ as defined in \cite{SZ}, which measures how curved the boundary of $\Omega$ is. For example, domains with flat boundary such as polygons have dimension $0$; domains with boundaries of non-vanishing curvature such as $C^2$ domains have dimension $1/2$. We recall the precise definition of $\kappa_\Omega$ in \S\ref{sec: background} below. We have the following result of Seeger--Ziesler on the boundedness of $B^\alpha_\Omega$.
\begin{theorem}[{Seeger--Ziesler \cite[Theorem 1.1]{SZ}}]\label{thm: seeger--ziesler}
  For all convex domains $\Omega$ with affine dimension $\kappa_\Omega$, the Bochner--Riesz means $B^\alpha_\Omega$ are bounded on $L^q(\R^2)$ for  
  \begin{align}
  \alpha>0\qquad&\text{when}\qquad 2\leq q\leq 4,\\\label{eqn:sz range}
  \alpha>\kappa_\Omega(1-4/q)\qquad&\text{when}\qquad 4\leq q\leq\infty.
  \end{align}
\end{theorem}
They also showed \cite[\S4]{SZ} that for all $\kappa\in[0,\frac{1}{2}]$, there exists a convex domain $\Omega$ with $\kappa_\Omega=\kappa$, for which \eqref{eqn:sz range} is sharp. 
However, Cladek \cite{cladek_BR} constructed certain domains for which \eqref{eqn:sz range} can be improved. 
\begin{theorem}[{Cladek \cite[Theorem 1.1]{cladek_BR}}]\label{thm: cladek}
    For all integer $m\geq 2$ and $\kappa\in(\frac{1}{4m+2},\frac{1}{4m-2}]$, there exists a convex domain $\Omega$ with $\kappa_\Omega=\kappa$ such that $B^\alpha_\Omega$ is bounded on $L^q(\R^2)$ for \begin{align}
        \alpha>\kappa_\Omega(1/2-2/q)\qquad&\text{when}\qquad4\leq q\leq 2m\label{eqn: cladek range a},\\
        \alpha>\kappa_\Omega\big(1-\frac{m+2}{q}\big)\qquad&\text{when}\qquad 2m\leq q\leq\infty.\label{eqn: cladek range b}
    \end{align}
\end{theorem}
By interpolating with a better $L^{2m}$ bound, Theorem \ref{thm: cladek} provides quantitative improvements over Theorem \ref{thm: seeger--ziesler} in the range $4<q<\infty$ for certain convex domains. We prove the following.
\begin{theorem}\label{thm: main}
Fix an integer $m\geq 2$ and an $\epsilon>0$. Then for all $\kappa\in(\frac{1}{2m+2},\frac{1}{2m}]$, there exists a convex domain $\Omega$ with $\kappa_\Omega=\kappa$, such that $B^\alpha_\Omega$ is bounded on $L^q(\R^2)$ for
\begin{align}
    \alpha>\kappa_\Omega(1/2-2/q)+\epsilon\qquad&\text{when}\qquad4\leq q\leq 2m,
   \label{eqn: main a} \\
    \alpha>\kappa_\Omega(1/2-5/2q+1/4m)+\epsilon\qquad&\text{when}\qquad 2m\leq q\leq 6m,\label{eqn: main b}\\
    \alpha>\kappa_\Omega\big(1-\frac{3m+1}{q}\big)+\epsilon\qquad &\text{when}\qquad 6m\leq q\leq\infty\label{eqn: main c}.
\end{align}
\end{theorem}

By interpolating with a better $L^{2m}$ as well as $L^{6m}$ bound, Theorem \ref{thm: main} also provides quantitative improvements over Theorem \ref{thm: seeger--ziesler} in the range $4<q<\infty$ for certain convex domains. The domains considered by us are different from the ones considered by Cladek \cite{cladek_BR} and so Theorem \ref{thm: main} is not a direct improvement to the bounds \eqref{eqn: cladek range a}, \eqref{eqn: cladek range b}.\footnote{Although the techniques used to obtain the $L^{6m}$ bound in our example can be used to directly improve the bounds in Cladek's example.} Rather, Theorem \ref{thm: main} shows that for each domain considered by Cladek \cite{cladek_BR} there exists a domain of the same dimension with a wider range of $L^p$-boundedness for the associated Bochner--Riesz operator. This can be seen by replacing $m$ by $2m-1$ in Theorem \ref{thm: main}: see Figure \ref{fig:SZ vs Cladek vs Us}.
\begin{figure}[h]
    \centering
    \makebox[0pt]{\includegraphics[width=17cm]{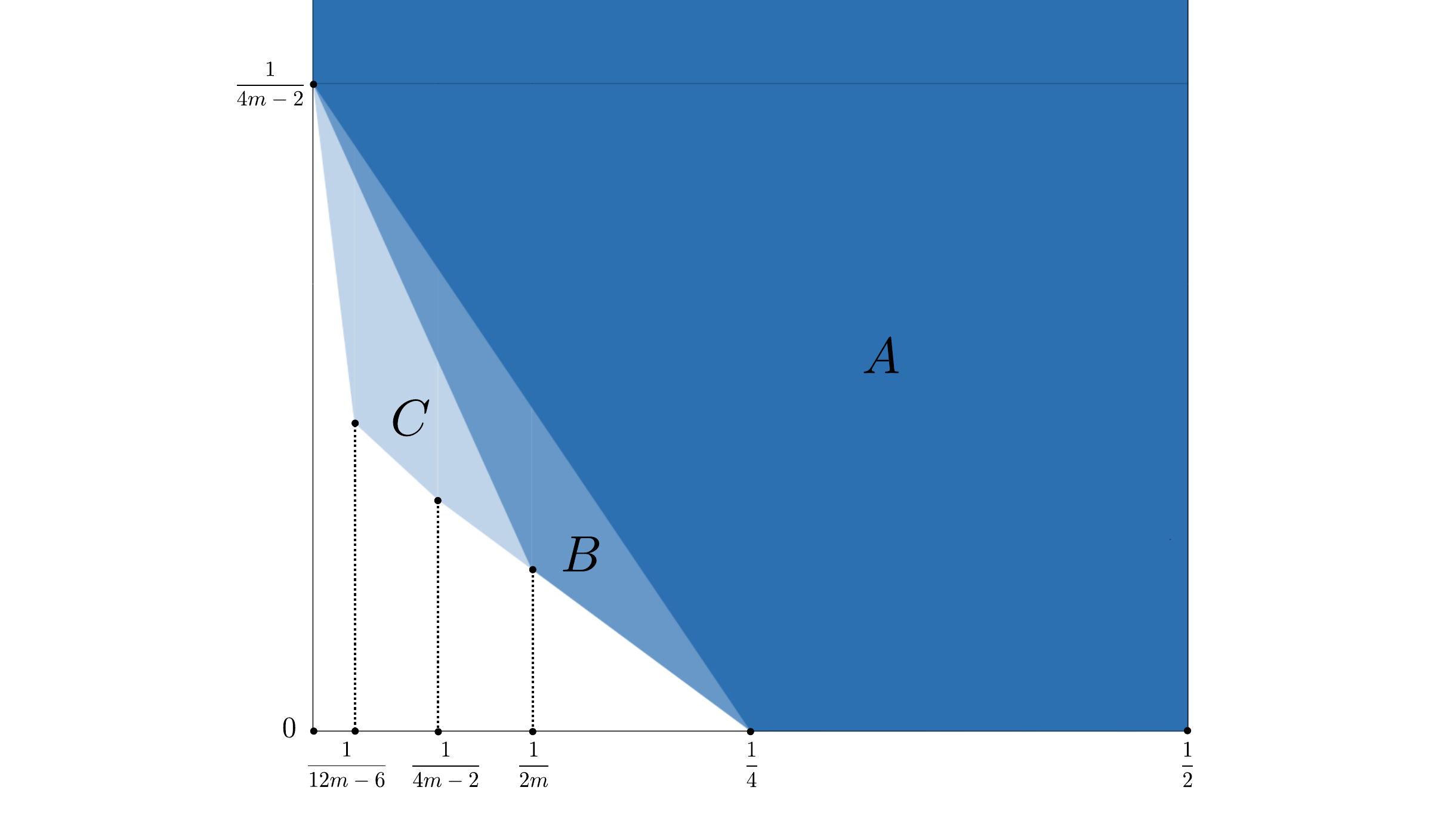}}
    \caption{A comparison of Seeger--Ziesler's \cite{SZ} range (A), Cladek's \cite{cladek_BR} example (B), and our example (C). Here we fix a dimension $\kappa=\frac{1}{4m-2}$ for some integer $m\geq 2$. Theorem \ref{thm: seeger--ziesler} ensures that $B^\alpha_\Omega$ is bounded on $L^q(\R^2)$ whenever $(1/q,\alpha)\in A$ for \textit{any} convex domain $\Omega$ satisfying $\kappa_{\Omega}=\kappa$. Theorem \ref{thm: cladek} shows there is a \textit{specific} convex domain $\Omega$ with $\kappa_\Omega=\kappa$ such that $B^\alpha_\Omega$ is bounded on $L^q(\R^2)$ whenever $(1/q,\alpha)\in A\cup B$. Theorem \ref{thm: main} improves this further by constructing another specific domain $\Omega$ with $\kappa_\Omega=\kappa$ such that $B^\alpha_\Omega$ is bounded on $L^q(\R^2)$ whenever $(1/q,\alpha)\in A\cup B\cup C$.}
    \label{fig:SZ vs Cladek vs Us}
\end{figure}

Theorem \ref{thm: cladek} has a simpler qualitative consequence: there exist domains of every dimension in $(0,\frac{1}{6})$ for which Theorem \ref{thm: seeger--ziesler} is not sharp. Theorem \ref{thm: main} broadens the range of possible dimensions to $(0,\frac{1}{4})$. We prove the following theorem which further improves this.
\begin{theorem}\label{thm: lambda p main}
Fix $p>2$ and an $\epsilon>0$. There exists a convex domain $\Omega$ with $\kappa_{\Omega}=\frac{1}{p}$, such that $B^\alpha_{\Omega}$ is bounded on $L^q(\R^2)$ for
\begin{align}
    \alpha>\kappa_\Omega\bigg(\frac{1}{2}-\frac{1}{3p}\bigg)\frac{(\frac{1}{4}-\frac{1}{q})}{(\frac{1}{4}-\frac{1}{3p})}+\epsilon\qquad&\text{when}\qquad 4\leq q\leq 3p,\label{eqn: lambda p main a}\\
    \alpha>\kappa_\Omega\bigg(1-\frac{3p+2}{2q}\bigg)+\epsilon\qquad &\text{when}\qquad 3p\leq q\leq\infty.\label{eqn: lambda p main b}
\end{align}
\end{theorem}

We remark that, by duality, all of the above theorems have counterpart results in the range $1 \leq q \leq 2$.

Theorem \ref{thm: lambda p main} shows there exist domains of every dimension in $(0,\frac{1}{2})$ for which Theorem \ref{thm: seeger--ziesler} is not sharp. In terms of the quantitative estimates, Theorem \ref{thm: lambda p main} is strictly weaker than Theorem \ref{thm: main} when $p\in 2\mathbb{N}$. For $p\notin 2\mathbb{N}$, Theorem \ref{thm: lambda p main} is stronger for large $q$ and weaker for small $q$.
\pagebreak
\subsection{Ideas in the proof}
There are three distinct ideas at play. 
\begin{itemize}[leftmargin=*]
\setlength\itemsep{1em}
    \item \textbf{Improved constructions:} Seeger--Ziesler \cite{SZ} show that every convex domain in $\R^2$ satisfies a `biorthogonality' property leading to a \textit{universal} $L^4$ square function estimate. This is similar to the classical C\'ordoba--Fefferman square function argument (\textit{cf.} \cite{Cordoba,fefferman_biorthogonality}). The $L^4$ square function estimate of Seeger--Ziesler leads to a \textit{universal} $L^4$ bound for $B^\alpha_\Omega$ in Theorem \ref{thm: seeger--ziesler}. For all integer $m\geq 2$, Cladek \cite{cladek_BR} shows that certain $\frac{1}{4m-2}$-dimensional domains exhibit a stronger `$m$-orthogonality' property which leads to an $L^{2m}$ square function estimate, and consequently the $L^{2m}$ bound in Theorem \ref{thm: cladek}. She provides an explicit example by constructing a \textit{generalized Sidon set} or \textit{$B_m$ set}\footnote{We recall the definition in \S\ref{sec: background}.} on the real line, and lifting it to the parabola to construct a domain $\Omega_{2m}$. $B_m$ sets are discrete sets on the real line that satisfy a one-dimensional $m$-orthogonality property. As a result, $\Omega_{2m}$ will satisfy a two-dimensional $m$-orthogonality property. The $B_m$ set constructed by Cladek is similar to Mian--Chowla's \cite{mian--chowla} construction of $B_2$ sets using a greedy algorithm. Such constructions are, however, sub-optimal in the sense that they yield $B_m$ sets that are sparse. Exploiting the optimal construction of Bose--Chowla \cite{bc}, we obtain $B_m$ sets that are more dense, which allows us to improve the dimensions of such convex domain $\Omega_{2m}$ from $\frac{1}{4m-2}$ to $\frac{1}{2m}$. Thus, by improving the construction in \cite{cladek_BR}, we are able to raise the dimension of the domain appearing in Theorem \ref{thm: cladek}. This idea is used in Theorem \ref{thm: main}.

    \item \textbf{Two-dimensional approach:} In \cite{cladek_BR} the author comments that her argument exploits orthogonality in only one dimension, and that a two-dimensional approach could yield better results. We achieve this by considering decoupling estimates for our domain. The one-dimensional $m$-orthogonality yields a one-dimensional $\ell^2 L^{2m}$ decoupling estimate for the underlying subset of the real line. Using a result of Chang \textit{et al.} \cite{chang}, this leads to a two-dimensional $\ell^2 L^{6m}$ decoupling estimate for our convex domain. If we think of decoupling estimate as a weaker version of orthogonality, then the result of Chang \textit{et al.} \cite{chang} says that a one-dimensional (weak) $m$-orthogonality implies a two-dimensional (weak) $3m$-orthogonality. Compare this to the argument above, where we used a one-dimensional $m$-orthogonality to obtain a two-dimensional $m$-orthogonality. By using the $\ell^2 L^{6m}$ decoupling estimate in conjunction with the $L^{2m}$ square function estimate mentioned above, we are able to get improvement over the $L^p$ bounds in Theorem \ref{thm: cladek}. We see this improvement in the statement of Theorem \ref{thm: main}.
    
    \item \textbf{$\Lambda(p)$ sets:} Decoupling estimates allow us to get improvements over the universal bounds of Theorem \ref{thm: seeger--ziesler} in the absence of a strong square function estimate. For instance, $\lamp$ sets are discrete sets on the real line that satisfy a version of orthogonality that is weaker than $m$-orthogonality. Domains constructed using $\lamp$ sets satisfy strong decoupling estimates without necessarily satisfying a strong square function estimate. By working with $\Lambda(p)$ sets, we are able to produce a wider class of domains for which Theorem \ref{thm: seeger--ziesler} is not sharp. This idea is used in Theorem \ref{thm: lambda p main}.
\end{itemize}
    \subsection{Notation}\label{sec: notation}
    For $n\in\mathbb{N}$ we denote $[n]:=\{1,2,\dots,n\}$. Unless the base is specified, all logarithms are taken base $2$. For $X,Y\in\mathbb{R}$, if there exists an absolute constant $C>0$ such that $X\leq CY$, we shall write $X\lesssim Y$ or equivalently $Y\gtrsim X$. If this constant $C$ is not absolute, but depends on some parameters of our problem, say $\alpha,\beta,\gamma$, we shall specify this by writing $X\lesssim_{\alpha,\beta,\gamma}Y$ or equivalently $Y\gtrsim_{\alpha,\beta,\gamma} X$. We also write $X\approx Y$ when both $X\lesssim Y$ and $X\gtrsim Y$ hold. Similarly, we write $X\approx_{\alpha,\beta\gamma}Y$ when both $X\lesssim_{\alpha,\beta,\gamma} Y$ and $X\gtrsim_{\alpha,\beta,\gamma} Y$ hold. For a compact interval $I$, we denote its centre by $c_I$, the left endpoint by $l_I$, and the right endpoint by $r_I$; so $I=[l_I,r_I]$ and $c_I:=(l_I+r_I)/2$. For a finite set $A$, we denote its cardinality by $\mathrm{card}(A)$.
    
    \subsection*{Acknowledgements} I would like to thank my advisor, Jonathan Hickman, for his guidance and valuable insights while working on this project.  
\section{Background}\label{sec: background}
In this section, we present some background on convex domains, generalized Sidon sets, $\Lambda(p)$ sets, and decoupling estimates.
\subsection{Affine dimension of convex domains}
Let $\Omega$ be a convex domain as defined in \S\ref{sec: main result}. We recall the definition of $\kappa_\Omega$ from \cite{SZ}. For $\delta>0$, we define a $\delta$-cap to be a collection of points on the boundary $\partial\Omega$ of the form $$B(\ell,\delta):=\{P\in\partial\Omega: \text{dist}(P,\ell)<\delta\},$$ where $\ell$ is any supporting line for $\partial\Omega$ (see Figure \ref{caps}). 
\begin{figure}[ht]
    \centering
    \includegraphics[width=9cm]{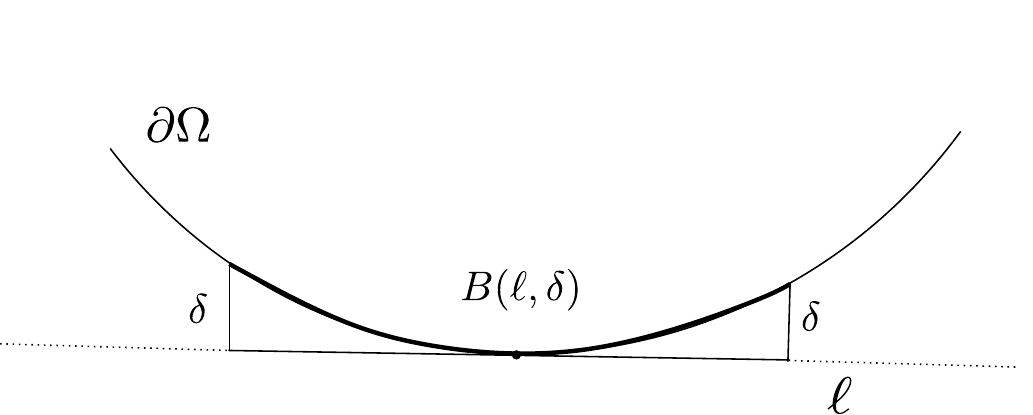}
    \caption{A $\delta$-cap of a domain $\Omega$ (drawn in bold).}
    \label{caps}
\end{figure}
Let $N(\Omega,\delta)$ be the minimum number of $\delta$-caps required to cover $\partial\Omega$. Then $\kappa_\Omega$ is defined as $$\kappa_\Omega:=\limsup_{\delta\to 0^+}\frac{\log{N(\Omega,\delta)}}{\log{\delta^{-1}}}.$$ 
Thus, for all $\epsilon>0$ and $0<\delta<1/2$, there exists a constant $c_\epsilon>0$ such that
 $$N(\Omega,\delta)\leq c_{\epsilon}\delta^{-\kappa_\Omega-\epsilon}.$$   For brevity, we shall often refer to $\kappa_\Omega$ as simply the \textit{dimension} of $\Omega$. 
From \cite[Lemma 2.3]{SZ} we see that $\kappa_\Omega\in[0,\frac{1}{2}]$. The dimension of any convex polygonal domain is $0$, and the dimension of any convex $C^2$ domain is $1/2$. The domains we construct will have dimensions strictly between $0$ and $1/2$.
\subsection{Additive energy of convex domains} Let $\Omega$ be a convex domain as above. We recall the definition of the \textit{$m$-additive energy} of $\Omega$ from \cite{cladek_BR}.  Let $\mathfrak{B}_\delta$ be a collection of $N(\Omega,\delta)$ caps at scale $\delta$ covering $\partial\Omega$. Consider a partition of $\mathfrak{B}_\delta$ as 
\begin{equation}
    \label{eqn: partition of caps}\mathfrak{B}_\delta=\bigsqcup_{i=1}^{M_0}\mathfrak{B}_{\delta,i},
\end{equation}
with the following property 

\begin{equation}
    \label{eqn: finite overlap property}
    \forall i,\text{ each $\xi\in\R^2$ lies in at most $M_1$ sets in the class }\{B_{1,i}+\dots+B_{m,i}:B_{j,i}\in\mathfrak{B}_{\delta,i}\}. 
\end{equation}
Let $$\Xi_m(\Omega,\delta):=\min\{M_0^{2m}\cdot M_1:\text{there is a partition }\eqref{eqn: partition of caps}\text{ of $\mathfrak{B}_\delta$ satisfying }\eqref{eqn: finite overlap property}\}.$$
Define the $m$-additive energy of $\Omega$ to be $$\mathcal{E}_{m,\Omega}:=\limsup_{\delta\to 0^+}\frac{\log\Xi_m(\Omega,\delta)}{\log{\delta^{-1}}}.$$
Thus, for all $\epsilon>0$ and $0<\delta<1/2$, there exists a constant $d_\epsilon>0$ such that $$\Xi_m(\Omega,\delta)\leq d_\epsilon\delta^{-\mathcal{E}_{m,\Omega}-\epsilon}.$$
We state the following result on the boundedness of $B^\alpha_\Omega$ in terms of the dimension and additive energy of $\Omega$.
\begin{theorem}[{\cite[Theorem 1.4]{cladek_BR}}]\label{thm: bochner riesz bounds}
    Let $\Omega$ be a convex domain and $m\geq 2$ be an integer. Then for $2m\leq p\leq\infty$, the Bochner--Riesz operator $B^\alpha_\Omega$ is bounded on $L^p(\R^2)$ for all $\alpha>\mathcal{E}_{m,\Omega}/p+\kappa_\Omega(1-(m+2)/p)$. 
\end{theorem}
It was shown by Seeger--Ziesler \cite[Lemma 2.4]{SZ} that $\mathcal{E}_{2,\Omega}=0$ for all convex domains $\Omega$, using a generalization of C\'ordoba's \cite{Cordoba} biorthogonality argument for the parabola. This recovers Theorem \ref{thm: seeger--ziesler}, when combined with Theorem \ref{thm: bochner riesz bounds}.
However, there are domains $\Omega$ for which $\mathcal{E}_{m,\Omega}> 0$ for $m>2$, which provide sharp examples for Theorem \ref{thm: seeger--ziesler}.  The reader may find these examples in \cite[\S4]{SZ} or \cite[\S1]{cladek_BR}. The examples constructed by Cladek satisfy $\mathcal{E}_{m,\Omega}=0$ for some $m>2$ \cite[Lemma 3.2]{cladek_BR}. Using Theorem \ref{thm: bochner riesz bounds}, this leads to the improved bounds in Theorem \ref{thm: cladek}. For Theorem \ref{thm: main} we construct domains that satisfy $\mathcal{E}_{m,\Omega}\leq m\epsilon$ for some small $\epsilon>0$. Using Theorem \ref{thm: bochner riesz bounds}, this will provide the estimate \eqref{eqn: main a} in Theorem \ref{thm: main}. We work with domains of low additive energy which are more amenable to induction on scale arguments, as opposed to domains of zero additive energy constructed by Cladek. As a trade off, we lose an $\epsilon$ in our estimates, which is reflected in our estimates \eqref{eqn: main a}, \eqref{eqn: main b} and \eqref{eqn: main c}.  
\subsection{Generalized Sidon sets}\label{sec: sidon}
Here we discuss the notion of \textit{generalized Sidon sets}, from additive combinatorics. 
\begin{definition}
    Fix $m\geq 2$. A set $A\subset\mathbb{N}_0$ is said to be $B_m^*[g]$, if $$\mathrm{card}\{(a_1,\dots,a_m)\in A^m:a_1+\dots+a_m=n\}\leq g\quad\text{for all}\quad n\in\mathbb{N}_0.$$
\end{definition}
This is related to the notion of $B_m[g]$ sets which are defined as follows. A set $A\subset\mathbb{N}_0$ is said to be a $B_m[g]$ set, if $$\mathrm{card}\{(a_1,\dots,a_m)\in A^m:a_1+\dots+a_m=n,a_1\leq\dots\leq a_m\}\leq g\quad\text{for all}\quad n\in\mathbb{N}_0.$$ 
These sets were introduced in 1932 by Sidon \cite{sidon} for the special case $m=2$ and $g=1$. He was interested in studying functions on the torus whose frequencies lie in these sets.  If $g=1$, then a $B_m[g]$ set is called a $B_m$ set. The sets $B_2$ are now called \textit{Sidon sets}, and $B_m[g]$ sets are often referred to as \textit{generalized Sidon sets}. Clearly every $B_m[g]$ set is a $B_m^*[g\cdot m!]$ set, but the converse may not hold. This is because the solutions to the equation $a_1+\dots+a_m=n$ can be permuted fewer than $m!$ times when the integers are not distinct. 

\begin{definition}
    For $m\geq 2$, and $N\geq 1$, let $F_{m,g}(N)$ denote the size of the largest $B_m^*[g]$ subset of $[N]$.
\end{definition}
A simple counting argument shows that 
\begin{equation}\label{eqn: F upper bound}F_{m,g}(N)\leq m^{1/m}(gN)^{1/m}.\end{equation}
To see this, let $A\subset[N]$ be a $B_m^*[g]$ set. Now $$|A|^m=|A^m|=\sum_{n=1}^{mN}|\{(a_1,\dots,a_m)\in A^m:a_1+\dots+a_m=n\}|.$$ By definition, $|\{(a_1,\dots,a_m)\in A^m:a_1+\dots+a_m=n\}|\leq g$ for all $n\in\mathbb{N}_0$. From this it follows that $$|A|^m\leq mNg,$$ which proves \eqref{eqn: F upper bound}.

The following well-known result provides a lower bound for $F_{m,g}(N)$.
\begin{theorem}[\textit{cf.} \cite{Lindstrom}]\label{theorem: bmg lower bound} For all $m\geq 2$, there exists a constant $c_m>0$ such that 
\begin{equation}\label{eqn: F lower bound} F_{m,g}(N)\geq c_m(gN)^{1/m},\end{equation}
for all $N,g$ sufficiently large. 
\end{theorem}
The $g=1$ case was established in 1962 by Bose--Chowla \cite{bc}. Theorem \ref{theorem: bmg lower bound} is proved by `gluing' together scaled copies of Bose and Chowla's $B_m$ sets. 
In light of the upper bound \eqref{eqn: F upper bound}, we find that \eqref{eqn: F lower bound} is sharp up to a constant factor depending on $m$. It is an active problem in the field of additive combinatorics to optimize this dependence on $m$. However, it is not relevant for our purposes. For more on the subject, we recommend the survey papers \cite{plagne,o'bryant}.
\begin{lemma}\label{lemma: sidon plus 1}   
Suppose $A$ is a $B_m^*[g]$ set and let $b\in\mathbb{N}\setminus A$. Then the set $A\cup\{b\}$ is a $B_m^*[1+m+(m-1)g]$ set.
\end{lemma}
\begin{proof}
    Let $B:=A\cup\{b\}$, and let $n\in\mathbb{N}_0$. We wish to count the number of tuples $(a_1,\dots,a_m)\in B^m$ satisfying \begin{equation}\label{eqn: bmg defining eqn} a_1+\dots+a_m=n.\end{equation}
    We have at most $g$ solutions coming from tuples of the form $(a_1,\dots,a_m)\in A^m$, since $A$ is a $B_m^*[g]$ set. More generally, we can fix $0\leq j\leq m$ and consider tuples $(a_1,\dots,a_m)\in B^m$ where $j$ entries of the tuples are equal to $b$. First let $0\leq j\leq m-2$. Let us rename the integers so that $a_1',\dots,a_{m-j}'\in A$ and $a_{m-j+1}'=\dots=a_m'=b$. Then $$a_1+\dots+a_m=n\iff a_1'+\dots+a_{m-j}'+jb=n.$$ Now the latter equation can be rewritten as \begin{equation}\label{eqn: bmg equation} a_1'+\dots+a_{m-j}'+ja=n+j(a-b),\end{equation} for some fixed $a\in A$. Since $A$ is a $B_m^*[g]$ set, the equation $$a_1'+\dots+a_m'=n+j(a-b)$$ has at most $g$ solutions  $(a_1',\dots,a_m')\in A^m$. Of these $g$ solutions, the tuples $(a_1',\dots,a_m')$ with $a_{m-j+1}'=\dots=a_m'=a$, are in a one-to-one correspondence with the solutions of \eqref{eqn: bmg equation}. The solutions of \eqref{eqn: bmg equation} are in turn in a one-to-one correspondence with solutions to the equation $$a_1'+\dots+a_{m-j}'+jb=n,$$ via $(a_1',\dots,a_{m-j}',a,\dots,a)\mapsto(a_1',\dots,a_{m-j}',b,\dots,b)$. 
It follows that \eqref{eqn: bmg defining eqn} has at most $g$ solutions $(a_1,\dots,a_m)\in B^m$ in this case. For $j=m-1$ and $j=m$, \eqref{eqn: bmg defining eqn} has at most $m$ and $1$ solution, respectively. Adding all the solutions for $j=0,1,\dots,m$ proves the result.
\end{proof}

\subsection{$\Lambda(p)$ sets}
We discuss some properties of $\Lambda(p)$ sets.
\begin{definition}
    For $A\subset\mathbb{Z}$ let $c_{00}(A)$ denote the linear space of all sequences indexed by $A$ with only finitely many non-zero terms. For $2\leq p\leq\infty$ define $$\|A\|_{\Lambda(p)}:=\sup\{\|\sum_{n\in A}a_ne^{2\pi inx}\|_{L^p([0,1])}:(a_n)_{n\in A}\in c_{00}(A)\text{ with } \sum_{n\in A}|a_n|^2\leq 1\}.$$
    The set $A$ is said to be a $\Lambda(p)$ set if $\|A\|_{\Lambda(p)}<\infty$.
\end{definition}
Any finite set $A\subset\mathbb{Z}$ satisfies $$\|A\|_{\Lambda(p)}\leq \mathrm{card}(A)^{1/2-1/p},$$ which follows from Plancherel's theorem, Cauchy--Schwarz, and the log-convexity of $L^p$ norms. In particular, every $A\subset [N]$ is $\Lambda(p)$ with $$\|A\|_{\Lambda(p)}\leq N^{1/2-1/p}.$$ We are interested in large sets $A\subset[N]$ which satisfy $$\|A\|_{\lamp}\lesssim_p 1.$$ The following result, when combined with Theorem \ref{theorem: bmg lower bound}, provides an answer for the special case $p\in 2\mathbb{N}$.
\begin{prop}\label{prop: bm implies lambda p}
   Let $A\subset\mathbb{Z}$ be a $B_m[g]$ set. Then $$\|A\|_{\Lambda(2m)}\leq(g\cdot m!)^{1/2m}.$$
\end{prop}

This result is well-known (see, for instance \cite{rudin}); nevertheless, for completeness, we include the short proof. 

\begin{proof}[Proof of Proposition~\ref{prop: bm implies lambda p}]
    Let $A$ be a $B_m[g]$ set and let $(a_n)_{n\in A}\in c_{00}(A)$.
    We have $$\big(\sum_{n\in A}a_n e^{2\pi i nx}\big)^m=\sum_{n\in A}b_ne^{2\pi inx},$$ where \begin{equation}\label{eqn: bn definition}
        b_n:=\sum_{\substack{n_1\dots,n_m\in A \\ n_1 + \cdots + n_m = n}}a_{n_1}\dots a_{n_m}.
    \end{equation} By Plancherel's theorem we have $$\bigg\|\sum_{n\in A}a_n e^{2\pi inx}\bigg\|_{L^{2m}([0,1])}^m=\big(\sum_{n\in A}|b_n|^2\big)^{1/2}.$$ Since $A$ is a $B_m[g]$ set, the sum on the right-hand side of \eqref{eqn: bn definition} has at most $g\cdot m!$ summands. Thus by Cauchy--Schwarz we have $$|b_n|\leq (g\cdot m!)^{1/2}\bigg(\sum_{n_1\dots,n_m\in A}|a_{n_1}\dots a_{n_m}|^2\bigg)^{1/2}.$$ 
    Now $\sum_{n_1,\dots,n_m\in A}|a_{n_1}\dots a_{n_m}|^2=\big(\sum_{n\in A}|a_n|^2\big)^m$,
    and so by the identity above we get $$\|\sum_{n\in A}a_n e^{2\pi inx}\|_{L^{2m}([0,1])}\leq (g\cdot m!)^\frac{1}{2m}\big(\sum_{n\in A}|a_n|^2\big)^{1/2},$$ which proves the claim.
\end{proof}
By Theorem \ref{theorem: bmg lower bound}, for $N$ sufficiently large, there exists a $B_m$ set $A\subset [N]$ with $\text{card}(A)\geq c_m N^{1/m}$. In light of Proposition \ref{prop: bm implies lambda p}, $$\|A\|_{\Lambda(2m)}\lesssim_m 1.$$ Thus, we have proved the following.
\begin{corollary}
    For all integer $m\geq 2$ and $N\in\mathbb{N}$ sufficiently large, there exists a set $S_N\subset [N]$ with $\mathrm{card}(S_N)\gtrsim_m N^{1/m}$ satisfying $$\|S_N\|_{\Lambda(2m)}\lesssim_m 1.$$
\end{corollary}
The following celebrated theorem due to Bourgain~\cite{Bourgain_lambdap} (see also \cite{Talagrand2014, Ryou2022}) is a generalization of the corollary above.
\begin{theorem}[{Bourgain \cite[Theorem 1]{Bourgain_lambdap}}]\label{thm: bourgain lambda p}
For all $p\geq 2$ and $N\in\mathbb{N}$ sufficiently large, there exists a set $S_N\subset[N]$ with $\card{S_N}\gtrsim_p N^{2/p}$ satisfying $$\|S_N\|_{\Lambda(p)}\lesssim_p 1.$$
\end{theorem}

The following result is easily verified. 
\begin{lemma} Let $A_1, A_2, A \subset\mathbb{Z}$ and $n \in \mathbb{Z}$. Then 
\begin{enumerate}[label=(\roman*)]
    \item $\|A_1\cup A_2\|_{\Lambda(p)}\leq\|A_1\|_{\lamp}+\|A_2\|_{\lamp}$;
    \item $\|n+ A\|_{\Lambda(p)} = \|A\|_{\Lambda(p)}$.
\end{enumerate}
\end{lemma}

\subsection{From points to intervals}
For technical reasons, we define the following weight functions. Given a bounded interval $Q\subset\R$, we define \begin{equation}
    \label{eqn: weight w_Q}
    w_Q(x):=\bigg(1+\frac{|x-c_Q|}{|Q|}\bigg)^{-10}.
\end{equation}
Before stating the main result of this subsection, we make the following observations. If $Q_x$ is an interval containing $x \in \R$ and $Q$ is a second bounded interval, by triangle inequality it follows that $$\big||x-c_Q|-|c_{Q_x}-c_Q|\big|\leq |Q_x|/2.$$ As such, if $|Q|=|Q_x|$, then $$w_Q(x)\approx w_Q(c_{Q_x}).$$
If $\mathcal{Q}$ is a collection of pairwise disjoint congruent intervals, then \begin{equation}\label{eqn: sum of weights is bounded}
\sum_{Q\in\mathcal{Q}}w_Q(x)\approx\sum_{Q\in\mathcal{Q}}w_Q(c_{Q_x})\leq\sum_{j=1}^\infty j^{-10}\lesssim 1\quad\text{for all $x\in\R$}.\end{equation} 

We state the following lemma, which is crucial for the construction of our domains. 

\begin{lemma}\label{lemma: *}
 For $p>2$ and all $N$ sufficiently large, there exists a family $\mathscr I(N;p)$ of $N$ subintervals of $[-1/2,1/2]$ satisfying the following properties. 
 \begin{enumerate}[label=(\roman*)]
     \item Each interval in $\I(N;p)$ has length $N^{-p/2}$.
     \item The intervals in $\I(N;p)$ are separated by $(p/4)N^{-p/2}$: $$\dist(I_1,I_2)\geq (p/4)N^{-p/2}\quad\text{for all}\quad I_1,I_2\in\I(N;p),\;I_1\neq I_2.$$
     \item There exists a constant $c_p>0$ such that for all intervals $Q$ of length $N^{p/2}$ we have $$\|\sum_{I\in\I(N;p)}f_I\|_{L^p(Q)}\leq c_p\big(\sum_{I\in\I(N;p)}\|f_I\|^2_{L^p(w_Q)}\big)^{1/2},$$ for all $f_I\in\mathcal{S}(\R)$ with Fourier support in $I\in\I(N;p)$. 
     \item The family $\I(N;p)$ contains the intervals $I^-:=[-1/2,-1/2+N^{-p/2}]$ and $I^+:=[1/2-N^{-p/2},1/2]$. 
 \end{enumerate}
 For the special case $p=2m$ with $m\in\mathbb{N}$, the family $\I(N;2m)$ can be chosen to also satisfy the following.
 \begin{enumerate}[label=(\roman*)]
 \setcounter{enumi}{4}
     \item There exists a constant $g_m\geq 1$, depending only on $m$, such that $$\max_{y\in\R}\;\;\mathrm{card}\{(I_1,\dots,I_m)\in \I(N;2m)^m:y\in I_1+\dots+I_m\}\leq g_m.$$
 \end{enumerate}
 \end{lemma}

We need two results in order to prove this lemma. The first is the following result from \cite{laba--wang}.
\begin{prop}[{\cite[Lemma 4]{laba--wang}}]\label{prop: laba wang}
    Let $A\subset\mathbb{Z}$ and define $E:=A+[-1/2,1/2]$. Then for all $f\in\mathcal{S}(\R)$ with Fourier support in $E$, and for all unit intervals $I$, we have $$\|f\|_{L^p(I)}\lesssim\|A\|_{\lamp}\|f\|_{L^2(\R)}.$$
\end{prop}

For completeness, we recall the proof from \cite{laba--wang}.

\begin{proof}[Proof of Lemma~\ref{prop: laba wang}]
    Let $P_E$ denote the Fourier projection operator $$(P_Ef)\;\widehat{}\;:=\chi_E\widehat{f}.$$
    It suffices to show that $$\|P_Ef\|_{L^p([0,1])}\lesssim\|A\|_{\lamp}\|f\|_{L^2(\R)}.$$The claim will then follow by translating the function $f$.
    
    Clearly $P_E$ is self-adjoint, so by duality we need to show that $$\|P_Ef\|_{L^2(\R)}\lesssim \|A\|_{\lamp}\|f\|_{L^{p'}([0,1])}.$$ By Plancherel, we write the left-hand side as $$\|P_Ef\|_{L^2(\R)}=\|\widehat{f}\|_{L^2(E)}.$$
    Now $E$ is the union of the essentially disjoint sets $\big\{n+[-1/2,1/2]:n\in A\big\}$, and so \begin{equation}\label{eqn: laba wang integral}
        \int_{E}|\widehat{f}|^2=\int_{-1/2}^{1/2}\sum_{n\in A}|\widehat{f}(n+x)|^2\dd{x}.
    \end{equation}
    Since $f\in \mathcal S(\R)$, the sequence $(\widehat{f}(n))_{n\in\mathbb{Z}}$ is in $\ell^1$. It follows that, $$\|\sum_{n\in A}e^{2\pi inx}\widehat{f}(n)\|_{L^p([0,1])}\leq\|A\|_{\Lambda(p)}(\sum_{n\in A}|\widehat{f}(n)|^2)^{1/2}\lesssim\|A\|_{\Lambda(p)}\|f\|_{L^2([0,1])}.$$ 
    Using duality again we get $$\|\sum_{n\in A}e^{2\pi inx}\widehat{f}(n)\|_{L^2([0,1])}\lesssim\|A\|_{\Lambda(p)}\|f\|_{L^{p'}([0,1])}.$$ By Plancherel's theorem this implies  $$\sum_{n\in A}|\widehat{f}(n)|^2\lesssim\|A\|^2_{\lamp}\|f\|^2_{L^{p'}([0,1])}.$$
    Since $\widehat{f}(n+x)=(e^{-2\pi ix(\,\cdot\,)}f)\;\widehat{}\;(n)$, by the above estimate applied to $e^{-2\pi ix(\,\cdot\,)}f$ we get $$\sum_{n\in A}|\widehat{f}(n+x)|^2\lesssim\|A\|_{\lamp}^2\|f\|_{L^{p'}([0,1])}^2\quad\text{for all}\quad x\in[-1/2,1/2].$$
    Using this in \eqref{eqn: laba wang integral} we find $$\int_E |\widehat{f}|^2\lesssim\|A\|_{\lamp}^2\|f\|_{L^{p'}([0,1])}^2,$$ which proves the claim.
\end{proof}
The second result we need in order to prove Lemma \ref{lemma: *} is the following consequence of Theorem \ref{thm: bourgain lambda p} and Theorem \ref{theorem: bmg lower bound}. 
\begin{corollary}\label{cor: lambda p}
   For all $p>2$ and $N\in\mathbb{N}$ sufficiently large, there exists $S_N\subset [N]$ with $\card{S_N}\geq (4p N)^{2/p}$ satisfying $$\|S_N\|_{\lamp}\lesssim_p 1.$$ 
   For the special case $p=2m$ for an integer $m\geq 2$, $S_N$ can be chosen to be a $B_m^*[g]$ set, for some $g\geq 1$ depending only on $m$. 
\end{corollary}
\begin{proof} We consider the cases $p\notin 2\mathbb{N}$ and $p\in 2\mathbb{N}$ separately.

\medskip\noindent\underline{$p\in (2,\infty)\setminus 2\mathbb{N}$:}\;\; If $S_N$ is as in the statement of Theorem \ref{thm: bourgain lambda p}, there is a constant $a_p>0$ such that $\text{card}(S_N)\geq a_p N^{2/p}$ for all $N$ sufficiently large. Let $k$ be chosen large enough, depending only on $p$, such that $a_pk^{1-2/p}\geq (8p)^{2/p}$. This is possible since $p>2$. 

Let $q:=\lfloor N/k \rfloor$. By Theorem \ref{thm: bourgain lambda p} there exists $S_q\subset [q]$ of size $\card{S_q}\geq a_pq^{2/p}$ satisfying $$\|S_q\|_{\lamp}\lesssim_p 1.$$ For $j=1,\dots, k$ define $T_j:=(j-1)q+S_q$. By translation invariance, $$\|T_j\|_{\lamp}\leq\|S_q\|_{\lamp}\lesssim_p 1\quad\text{for each}\quad j=1,\dots,k.$$ For $S_N:=\bigcup_{j=1}^k T_j$ it follows from the union lemma that  $$\|S_N\|_{\lamp}\leq \sum_{j=1}^k\|T_j\|_{\lamp}\lesssim_p k\lesssim_p 1.$$ Now $S_N\subset [N]$ and $\card{S_N}\geq ka_p q^{2/p}\geq(4p N)^{2/p}$ by our choice of $k$.

\medskip\noindent\underline{$p\in 2\mathbb{N}$:}\;\; Let $p=2m$ for some integer $m\geq 2$. Let $c_m>0$ be chosen small enough, so that \eqref{eqn: F lower bound} holds with $g:=8mc_m^{-m}$. Thus, \begin{equation}\label{eqn: F lower bound 2}F_{m,g}(N)\geq (8mN)^{1/m}.\end{equation}
In other words, there exists a $B_m^*[g]$ set $S_N\subset [N]$ with $\text{card}(S_N)\geq (8mN)^{1/m}$, where $g=8mc_m^{-m}$. By Proposition \ref{prop: bm implies lambda p}, we have $$\|S_N\|_{\Lambda(2m)}\leq (g\cdot m!)^{1/2m}\lesssim_m 1.$$
\end{proof}
Let us now prove Lemma \ref{lemma: *}.
 \begin{proof}[Proof of Lemma \ref{lemma: *}]
        Define $N_p:=\lceil\frac{N^{p/2}}{4p}\rceil$, $d_p:=N^{-p/2}N_p$, and note that $d_p\leq\frac{1}{2p}$ for sufficiently large $N$.  By Corollary \ref{cor: lambda p}, there exists a set $P'(N;p)\subset [N_p-1]$ of size $N-2$ satisfying $$\|P'(N;p)\|_{\lamp}\lesssim_p 1.$$ Then $$P(N;p):=P'(N;p)\cup\big\{0,N_p\big\}$$ is a set of $N$ points satisfying $$\|P(N;p)\|_{\lamp}\lesssim_p 1.$$
        In case $p=2m$ for an integer $m\geq 2$, the same corollary says that $P'(N;2m)$ can be chosen to be a $B_m^*[g]$ set for $g=8mc_m^{-m}$. By Lemma \ref{lemma: sidon plus 1}, $$P(N;2m):=P'(N;2m)\cup\big\{0,N_{2m}\big\}$$ is a $B_m^*[g_m]$ set of size $N$, where $g_m:=m+m^2+8m(m-1)^2c_m^{-m}$. 
     
     We use $P(N;p)$ to construct the auxiliary family of intervals $$I(N;p):=\big\{I_{x,p}:x\in P(N;p)\big\},$$ where $$I_{x,p}:=\begin{cases}[0,d_p],\quad &x=0,\\
   [x-\frac{d_p}{2},x+\frac{d_p}{2}],\quad &0<x<N_p,\\
   
   [N_p-d_p,N_p], \quad &x=N_p.
   \end{cases}$$ We shall obtain $\I(N;p)$ from $I(N;p)$ by rescaling.  
   
   Let $Q$ be an interval of length $d_p^{-1}$. There exists $\psi_Q\in\mathcal{S}(\R)$ with the following properties \begin{itemize}
       \item $\mathrm{supp}(\psi_Q)\;\widehat\;\subset [-1/4,1/4]$;
       \item $\psi_Q(t)\gtrsim 1$ for all $t\in Q$;
       \item $\psi_Q(t)\lesssim w_Q(t)$ for all $t\in\R$.
   \end{itemize}
   Then for all $f_I\in\mathcal{S}(\R)$ with Fourier support in $I\in\I(N;p)$, the function \\$\sum_{I\in I(N;p)}f_I\cdot \psi_Q$ is Fourier supported in $P(N;p)+[-1/2,1/2]$, and so, by Proposition \ref{prop: laba wang}, 
    $$\|\sum_{I\in I(N;p)}f_I\|_{L^p(Q)}\lesssim\|\sum_{I\in I(N;p)}f_I\cdot \psi_Q\|_{L^p(Q)}\lesssim_p\|\sum_{I\in I(N;p)}f_I\cdot \psi_Q\|_{L^2(\R)}.$$  We can decouple the right-hand side by Plancherel $$\|\sum_{I\in I(N;p)}f_I\cdot \psi_Q\|_{L^2(\R)}=\big(\sum_{I\in I(N;p)}\|f_I\cdot \psi_Q\|^2_{L^2(\R)}\big)^{1/2}.$$ Combining the two estimates above yields $$\|\sum_{I\in I(N;p)}f_I\|_{L^p(Q)}\lesssim_p\big(\sum_{I\in I(N;p)}\|f_I\cdot \psi_Q\|_{L^2(\R)}^2\big)^{1/2}.$$
     By H\"older's inequality we have $$\int_{\R}|f_I|^2|\psi_Q|^2\leq\bigg(\int_{\R}|f_I|^p|\psi_Q|\bigg)^{2/p}\bigg(\int_{\R}|\psi_Q|^{2\frac{p-1}{p-2}}\bigg)^{1-2/p}.$$
     Now $$\bigg(\int_{\R}|\psi_Q|^{2\frac{p-1}{p-2}}\bigg)^{1-2/p}\lesssim_p 1,$$ and so using the fact that $\psi_Q\lesssim w_Q$, we have \begin{equation}\label{eqn: weighted decoupling est}\|\sum_{I\in I(N;p)}f_I\|_{L^p(Q)}\lesssim_p\big(\sum_{I\in I(N;p)}\|f_I\|^2_{L^p(w_Q)}\big)^{1/2}.\end{equation}

      Define the collection of rescaled intervals $$\I(N;p):=\big\{-1/2+N_p^{-1} I:I\in I(N;p)\big\}.$$ By construction, $\I(N;p)$ satisfies  conditions (i), (ii), and (iv). Property (iii) follows by rescaling the inequality \eqref{eqn: weighted decoupling est}.

      Suppose $p=2m$ for an integer $m\geq 2$. We show that (v) holds for the family $\I(N;2m)$. It is easy to see that for all $x_1,\dots,x_{2m}\in P(N;2m)$, the sumsets $I_{x_1,m}+\dots+I_{x_m,m}$ and $I_{x_{m+1},m}+\dots+I_{x_{2m},m}$ can intersect only if $$|(x_1+\dots+x_m)-(x_{m+1}+\dots+x_{2m})|\leq 1/2.$$ But since $(x_1+\dots+x_m)-(x_{m+1}+\dots+x_{2m})\in\mathbb{Z}$, this can only happen if $x_1+\dots +x_m=x_{m+1}+\dots +x_{2m}$. Since $P(N;2m)$ is a $B_m^*[g_m]$ set, it follows that $$\mathrm{card}\{(I_1,\dots,I_m)\in I(N;m)^m:y\in I_1+\dots+I_m\}\leq g_m,$$
   for all $y\in\mathbb{R}$. The same holds for the family $\I(N;2m)$ since it is obtained by rescaling $I(N;2m)$.

 \end{proof}

\subsection{Decoupling estimates for Cantor sets} \label{sec: cantor def}
Here we recall certain decoupling estimates for \textit{generalized Cantor sets}, which were recently established by Chang \textit{et al.} \cite{chang}.

Let $\I$ be a family of subintervals of $[-1/2,1/2]$ of length $\lambda$. We first recursively construct a sequence $(\I_k)_{k=1}^\infty$ where $\I_k$ is a family of subintervals of $[-1/2,1/2]$ of length $\lambda^k$. To do this, let $\I_1:=\I$ and suppose we have defined $\I_1,\dots,\I_k$ for some $k$. For $I\in\I_k$, let $L_I$ be the order-preserving affine transformation mapping $I$ onto $[-1/2,1/2]$. We define $\I_{k+1}$ as $$\I_{k+1}:=\{L_I^{-1}(J)\subset I:J\in\I,\; I\in\I_k\}.$$ We define the \textit{$k$th level} $C_k(\I)$ of $C(\I)$ to be $C_k(\I):=\bigcup_{J\in\I_k}J$. By construction, we have $$C_1(\I)\supset C_2(\I)\supset\dots \supset C_k(\I)\supset C_{k+1}(\I)\supset\dots.$$
Finally, we define $C(\I):=\bigcap_{k=1}^\infty C_k(\I).$
\begin{definition}[Generalized Cantor sets]
The set $C(\I)$ defined above is called the generalized Cantor set associated with $\I$.
\end{definition}
We also define $\I_k'$ to be the collection of all connected components of $C_{k-1}(\I)\setminus C_k(\I)$. These are the intervals that have been discarded precisely at the $k$th iteration above. Such an interval will be called a \textit{removed interval}. On the other hand, any interval in the family $\bigcup_{k=1}^\infty \I_k$ shall be called a \textit{parent interval}. 

\begin{remark}
   There are some anomalies with the definition of a generalized Cantor set in \cite{chang}; however, the intended definition is equivalent to the one stated above.
\end{remark}
Before stating the results on decoupling, we introduce some notation. Let $C=C(\I)$ be a generalized Cantor set, and $\I_k$ be the family of intervals forming its $k$th level $C_k=C_k(\I)$. Define $N(k):=\card{\I_k}$ to be the number of intervals forming $C_k$. For $0\leq j\leq k$, and every $I\in\I_j$, we define $\I_k(I):=\{J\in\I_k:J\subset I\}$. We use $\text{dim}(C)$ to denote the Minkowski dimension of $C$.
For an interval $I\in\I_k$  (or more generally any $I\subset[-1/2,1/2]$), we define the $|I|\times |I|^2$ parallelogram $$\theta_I:=\{\xi\in\R^2:\xi_1\in I,\;|\xi_2-2c_I(\xi_1-c_I)-c_I^2|\leq|I|^2/2\},$$ where $c_I$ denotes the centre of $I$. The parallelogram $\theta_I$ is comparable to a $|I|^2$-neighbourhood of the piece of the parabola $(t,t^2)$ above the interval $I$. 
We denote by $D_{q,k}(C)$ the infimum of all constants $D$, such that for all $f_{\theta_J}\in\mathcal{S}(\R^2)$  with Fourier support in the parallelogram $\theta_J$, we have
\begin{equation*}
    \label{eqn: decoupling constant}
    \|\sum_{J\in\I_k}f_{\theta_J}\|_{L^q(\R^2)}\leq D\big(\sum_{J\in\I_k}\|f_{\theta_J}\|_{L^q(\R^2)}^2\big)^{1/2}.
\end{equation*}
The following theorem from \cite{chang} compares the two-dimensional $\ell^2 L^{3q}$ decoupling constant with the one-dimensional $\ell^2 L^q$ decoupling constant for $C$. 
\begin{theorem}[{\cite[Theorem 1.1]{chang}}]\label{thm: chang}
For $q\geq 2$, let $\nu_q(C)$ be the smallest exponent such that for all $\epsilon>0$ we have
\begin{equation*}
    \|\sum_{J\in\I_k}f_J\|_{L^q(\R)}\lesssim_{q,\epsilon,\mathrm{dim}(C),N(1)}N(k)^{\nu_q(C)+\epsilon}\big(\sum_{J\in\I_k}\|f_J\|_{L^q(\R)}^2\big)^{1/2},
\end{equation*}
for all $k\geq 0$ and all $f_J\in\mathcal{S}(\R)$ with Fourier support in the interval $J$. Then for all $\epsilon>0$ we have $$D_{3q,k}(C)\lesssim_{q,\epsilon,\mathrm{dim}(C),N(1)}N(k)^{\nu_q(C)+\epsilon}.$$
\end{theorem}
Theorem \ref{thm: chang} is a generalization of the Bourgain--Demeter \cite{bourgain--demeter} decoupling estimates for the parabola, and is proved using arguments similar to those from \cite{shortdecoupling,wooley}. 
\section{Constructing the domain}\label{sec: construction} 
We use Lemma \ref{lemma: *} to construct the domains for our main results. Let $\epsilon$ be a positive number, which we shall keep fixed from now on. 

\subsection{Construction}\label{sec: construction subsection}  Let $p>2$ and $N\geq 1$ be sufficiently large (to be specified later). Recall the family $\I(N;p)$ from Lemma \ref{lemma: *}. Let $\I^p:=\I(N;p)$ and let $C^p:=C(\I^p)$ be the generalized Cantor set associated with it. Let $C^p_k:=C_k(\I^p)$ be the $k$th level of $C^p$, which is formed by the family of intervals $\I^p_k$ of length $N^{-pk/2}$. Since the intervals in $\I(N;p)$ are separated by $(p/4)N^{-p/2}$, the $k$th generation of removed intervals have length at least $(p/4)N^{-pk/2}$. Note that the condition $I^-,I^+\in\I(N;p)$ of Lemma \ref{lemma: *} implies that $C^p$ contains the endpoints of every parent interval. We define our domains as $$\Omega_p:=\text{int}(\Omega'),\quad\text{with}\quad\Omega':=\text{conv}\{(t,t^2-1/8):t\in C^p\},$$ where we write $\text{conv}(A)$ to denote the convex hull of the set $A$, and  $\text{int}(A)$ to denote the interior of the set $A$.

Let us decompose the boundary as $\partial\Omega_p=\mathfrak{G}_1\cup\mathfrak{G}_2$, where $\mathfrak{G}_1:=\{\xi\in\partial\Omega_p:\xi_2<1/8\}$, and $\mathfrak{G}_2:=\{\xi\in\partial\Omega_p:\xi_2= 1/8\}$ It is clear that $\mathfrak{G}_2$ is a line segment, and that $\mathfrak{G}_1$ can be graph-parameterised over the interval $[-1/2,1/2]$, so that there exists a function $\gamma:[-1/2,1/2]\to\R$ such that $$\mathfrak{G}_1=\{(t,\gamma(t)):|t|\leq 1/2\}.$$
By the convexity of $\Omega_p$, it follows that $\gamma$ is a convex function.

When the context is clear, we shall often suppress the dependence on $p$ of the sets defined above. For instance, we may denote $C^p$ by simply $C$, $\I^p$ by simply $\I$ etc.
\subsection{Dimension} \label{sec: dimension and additive energy}
We are going to show that $\kappa_{\Omega_p}=\frac{1}{p}$.
Fix $\delta\in(0,1/2)$ and define $K(\delta)$ as $$K(\delta):=\min\{k\in\mathbb{N}:|I|<\delta^{1/2}\quad\text{for all}\quad I\in\I_k\}.$$
Since the intervals in $\I_k$ all have length $N^{-pk/2}$, we have \begin{equation}\label{eqn: N K(delta)}
N^{K(\delta)-1}\leq\delta^{-1/p}< N^{K(\delta)},\end{equation}
so that \begin{equation}
    \label{eqn: K(delta) bound}
\frac{1}{p}\frac{\log{\delta^{-1}}}{\log N}<  K(\delta)\leq 1+\frac{1}{p}\frac{\log{\delta^{-1}}}{\log N}.
\end{equation}
\begin{claim}\label{claim: cap estimate}
Let $N(\Omega_{p},\delta)$ be the minimum number of $\delta$-caps required to cover $\partial\Omega$. Then $$N(\Omega_{p},\delta)\approx_{m,N} N^{K(\delta)}.$$   
\end{claim}
Recall that $\partial\Omega=\mathfrak{G}_1\cup\mathfrak{G}_2$, where $\mathfrak{G}_2$ is a line segment. Thus, $\mathfrak{G}_2$ can be covered by a single $\delta$-cap, and as such we can focus our attention on $\mathfrak G_1$ only. 
Let $S(\delta)$ be the collection of the endpoints of every interval in $\I_{K(\delta)}$, and let $\I(\delta)$ be the family of essentially disjoint intervals in the decomposition of $[-1/2,1/2]$ given by the partition $S(\delta)$. We partition $\I(\delta)$ as $\I(\delta)=\Ir\sqcup\I_{K(\delta)}$, where $\Ir$ is the family of all removed intervals contained in $\I(\delta)$.
The portion of $\mathfrak G_1$ over each $I\in\Ir$ is a line segment. Using the notation from \S\ref{sec: notation} we can write $$\gamma(t)=l_I^2+(r_I+l_I)(t-l_I),\quad\text{for all}\quad t\in I=[l_I,r_I].$$
Then there exists a unique supporting line $\ell_I$ to $\mathfrak G_1$ at the point $(c_I,\gamma(c_I))$ with slope $l_I+r_I$. Define the $\delta$-cap $B_I:=B(\ell_I,\delta)$, which clearly includes the aforementioned line segment. 
Now let $I\in\I_{K(\delta)}$, and fix a supporting line $\ell_I$ to $\mathfrak G_1$ at the point $(c_I,\gamma(c_I))$, and define the $\delta$-cap $B_I:=B(\ell_I,\delta)$. We note that this supporting line may not be unique. In order to prove Claim \ref{claim: cap estimate}, it suffices to prove the following.
\begin{claim}\label{claim: near optimal covering}
    The family $\mathfrak{B}_\delta:=\{B_I:I\in\I(\delta)\}$ is a covering of $\mathfrak G_1$. Moreover, this covering is near-optimal in the sense that $$\mathrm{card}(\mathfrak B_\delta)\approx_{p,N} N(\Omega_{p},\delta).$$
\end{claim}
Indeed this implies Claim \ref{claim: cap estimate}, since $\text{card}(\mathfrak B_\delta)=\card{\I(\delta)}\approx N^{K(\dl)}$. To prove the latter claim we follow an argument similar to the one in \cite{SZ}. 
\begin{proof}[Proof of Claim \ref{claim: near optimal covering}]
For $t\in[-1/2,1/2]$ let $P_t:=(t,\gamma(t))$. For $t_0\in[-1/2,1/2]$, let $\ell_0$ be a supporting line at a point $P_{t_0}$ on the boundary $\mathfrak G_1$. Then for all $t>t_0$, a simple trigonometric argument shows that 
\begin{equation}\label{eqn: trig dist identity}
   \dist(P_t,\ell_0)=\frac{1}{\sqrt{1+\sigma^2}}(\gamma(t)-\gamma(t_0)-\sigma(t-t_0)), 
\end{equation}
\begin{figure}\
    \centering
    \includegraphics[width=10cm]{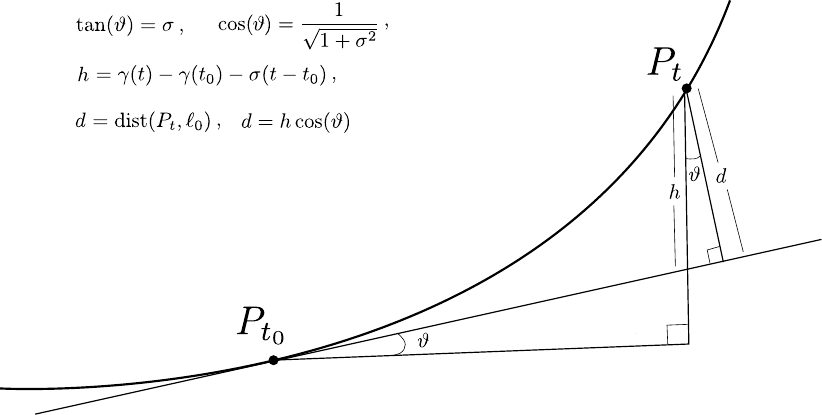}
    \caption{The distance formula \eqref{eqn: trig dist identity}.}
   \label{fig: trig distance identity}
\end{figure}
where $\sigma$ is the slope of $\ell_0$. For a proof see Figure \ref{fig: trig distance identity}. The convexity of $\gamma$ also implies that the distance function above is increasing in $t$. Let $\gamma_R'$ denote the right and $\gamma_L'$ denote the left derivative of $\gamma$. Then the right derivative of $\dist(P_t,\ell_0)$ with respect to $t$ is given by $$\frac{1}{\sqrt{1+\sigma^2}}(\gamma_R'(t)-\sigma).$$ Since $\sigma$ is the slope of a supporting line at $P_{t_0}$, we have $\gamma_L'(t_0)\leq\sigma\leq\gamma_R'(t_0)$. By the convexity of $\gamma$, we know that $\gamma_R'(t)\geq\gamma_R'(t_0)$ for all $t>t_0$. Hence for all such $t$, the right-hand derivative of $\dist(P_t,\ell_0)$ is non-negative, and thus $\dist(P_t,\ell_0)$ is non-decreasing. 
For each $I\in\Ir$, we observed that the $\delta$-cap $B_I$ covers the portion of the boundary above $I$, that is $I\subset \text{proj}_1(B_I)$, where $\text{proj}_1(A)$ denotes the projection of the set $A\subset\R^2$ onto the horizontal coordinate axis. Now let $I\in\I_{K(\delta)}$. If $I=[a,b]$ and $\sigma$ denotes the slope of the supporting line $\ell_I$, then by \eqref{eqn: trig dist identity} we have 
$$\dist(P_b,\ell_I)=\frac{1}{\sqrt{1+\sigma^2}}\bigg(\gamma(b)-\gamma\bigg(\frac{a+b}{2}\bigg)-\sigma\cdot\frac{b-a}{2}\bigg).$$
Since $b$ is an endpoint of a parent interval, we have $\gamma(b)=b^2$. Furthermore, we have $\gamma\big(\frac{a+b}{2}\big)\geq\big(\frac{a+b}{2}\big)^2$, and since there is a supporting line of slope $2a$ at $P_a$, by the convexity of the boundary it follows that $\sigma\geq 2a$. Therefore, $$\dist(P_b,\ell_I)\leq \frac{3}{4}(b-a)^2.$$ Since $|I|<\delta^{1/2}$ by the definition of $K(\delta)$, the monotonicity of the distance function implies that $$\dist(P_t,\ell_I)<\delta,\qquad\text{for all}\qquad c_I\leq t\leq b.$$ Similar arguments show that $$\dist(P_t,\ell_I)<\delta\qquad\text{for all}\qquad a\leq t\leq c_I.$$Thus for all $I\in\I_{K(\delta)}$, the $\delta$-cap $B_I$ covers the portion of the boundary over $I$, that is $I\subset \text{proj}_1(B_I)$. Since $\I(\delta)$ is a covering of the interval $[-1/2,1/2]$, it follows that $\{\text{proj}_1(B_I):I\in\I(\delta)\}$ is also a covering of $[-1/2,1/2]$. Therefore, $\mathfrak{B}_{\delta}$ is a covering of $\mathfrak G_1$ by $\delta$-caps. This establishes the first part of the claim. 

Since $\mathfrak B_\delta$ is a covering, we have 
$N(\Omega,\delta)\leq\text{card}(\mathfrak B_\delta).$
For the reverse inequality, let $\mathfrak B_\delta'$ be a covering of $\mathfrak G_1$ by $N(\Omega,\delta)$ number of $\delta$-caps. The number of removed intervals in $\I(\delta)$ is $N^{K(\delta)}-1$, and the lengths of these intervals are at least $(p/4)N^{-pK(\delta)/2}$. By definition of $K(\delta)$, we have $(p/4)N^{-(p/2)K(\delta)}\geq (p/4)N^{-p/2}\delta^{1/2}$. For each $I\in\Ir$, the supporting line $\ell_I$ has slope $l_I+r_I$. Now consider $\lceil N^{p}\rceil$ removed intervals $I_1,\dots,I_{\lceil N^{p}\rceil}\in\Ir$ with $c_{I_1}\leq\dots\leq c_{I_{\lceil N^{p}\rceil}}.$ By the distance formula \eqref{eqn: trig dist identity}, we get $$\dist(l_{2},\ell_1)=\frac{1}{\sqrt{1+(l_1+r_1)^2}}(l_2-r_1)(l_2-l_1),$$
where $\ell_1:=\ell_{I_1}, l_1:=l_{I_1},r_1:=r_{I_1},$ and $l_2:=l_{I_{\lceil N^{p}\rceil}}$. Then, $$\dist(l_2,\ell_1)\geq\frac{1}{\sqrt{2}}(l_2-r_1)^2.$$
Since each $I_j$ has length at least $(p/4)N^{-p/2}\delta^{1/2}$, it follows that $$\dist(l_2,\ell_1)>\delta.$$
Thus there exists at least one $\delta$-cap from $\mathfrak B_\delta'$ between $\lceil N^{p}\rceil$ consecutive $\delta$-caps  $B_I\in\mathfrak B_\delta$, for $I\in\I(\delta)\setminus\I_{K(\delta)}$. Hence, $\text{card}(\mathfrak B_\delta)\lesssim N^{p}\text{card}(\mathfrak B_\delta')$, which proves the claim. \end{proof}
Therefore, $N(\Omega,\delta)\approx_{p,N}N^{K(\delta)}$ and so by \eqref{eqn: N K(delta)}, we find that 
 \begin{equation}\label{eqn: dimension of omega}
 \kappa_{\Omega_p}=\frac{1}{p}.\end{equation}
 \begin{remark}\label{rmk: lower dimensions}
    We can modify the argument so that $\kappa_{\Omega_{2m}}$ is equal to any prescribed value $\kappa\in\big(\frac{1}{2m+2},\frac{1}{2m}]$, as stated in Theorem \ref{thm: main}. Observe that if \eqref{eqn: F lower bound 2} holds then $$F_{m,g}(N)\geq (4mN)^b$$ holds for all $0<b<\frac{1}{m}$. Thus, if we choose the $B_m^*[g]$ set in the proof of Lemma \ref{lemma: *} sub-optimally, the construction above will produce a domain $\Omega_{2m}$ with similar additive structure such that $\kappa_{\Omega_{2m}}=\kappa$. 
\end{remark}
\subsection{Additive energy}\label{sec: additive energy}
In the special case $p=2m\in 2\mathbb{N}$, we show that $\Omega_{2m}$ has the additional property $\mathcal{E}_{m,\Omega_{2m}}<m\epsilon$. 
If $\I_k'$ is as defined in \S\ref{sec: cantor def} for $k\geq 1$, and $\I_0':=\I_{K(\delta)}$, then $\I(\delta)$ can be partitioned as 
\begin{equation}\label{eqn: partition}\I(\delta)=\I_0'\cup\I_1'\cup\I_2'\cup\dots\cup\I_{K(\delta)}'.\end{equation}

Correspondingly, we get a partition of $\mathfrak{B}_\delta$ as
\begin{equation}\label{eqn: B_delta partition}\mathfrak{B}_\delta=\mathfrak B_{\delta,0}'\cup\mathfrak{B}'_{\delta,1}\cup\mathfrak{B}_{\delta,2}'\cup\dots\cup\mathfrak{B}_{\delta,K(\delta)}',\end{equation} 
where $\mathfrak{B}_{\delta,k}':=\{B_I:I\in\I_k'\}$ for $0\leq k\leq K(\delta)$.  

\begin{claim}\label{claim: claim 1}
     For all $0\leq k\leq K(\delta)$, and for every $\xi\in\R^2$, there exist at most $N^m\cdot g_m^{K(\delta)}$ tuples $(B_{1,k},\dots,B_{m,k})\in\mathfrak{B}_{\delta,k}'^m$ satisfying $$\xi\in B_{1,k}+\dots+B_{m,k}.$$ 
\end{claim}

Notice that each interval in $\I_k'$ is contained in some interval of $\I_{k-1}$, and each interval of $\I_{k-1}$ contains exactly $N-1$ interval of $\I_k'$. Thus, in order to prove Claim \ref{claim: claim 1}, we first prove the following. 

\begin{claim}
    \label{claim: claim 2}
   For all $k\geq 1$, and all $y\in\R$, we have $$\mathrm{card}\{(I_1,\dots,I_m)\in\I_k^m:y\in I_1+\dots+I_m\}\leq g_m^k.$$
\end{claim}\begin{proof}[Proof of Claim \ref{claim: claim 2}]
We prove this by inducting on $k$. For $k=1$ the claim holds by Lemma \ref{lemma: *} (v), since $\I_1=\I(N;2m)$. Suppose the claim is true for some $k\geq 1$. Fix a point $y\in\R$. To count the number of tuples $\{(I_1',\dots,I_m')\in\I_{k+1}^m:y\in I_1'+\dots+I_m'\}$, we first notice that each $I_r'$ is contained in a unique $I_r\in\I_k$. By the induction hypothesis, $$\mathrm{card}\{(I_1,\dots,I_m)\in\I_k^m:y\in I_1+\dots+I_m\}\leq g_m^{k-1}.$$ Now we fix $(I_1,\dots,I_m)\in\I_k^m$ and show that $$\mathrm{card}\{(I_1',\dots,I_m')\in\I_{k+1}^m:y\in I_1'+\dots+I_m',\;I_r'\subset I_r\}\leq g_m.$$ To prove this, let $x_r$ be the centre of the interval $I_r\in\I_k$, and suppose $y\in I_1'+\dots+I_m'$, where $(I_1',\dots,I_m')\in\I_{k+1}^m$ and $I_r'\subset I_r$. This can happen if and only if $$y'\in L_{I_1}(I_1')+\dots+L_{I_m}(I_m')\quad\text{where}\quad y':=N^{mk}(y-x_1-\dots-x_m).$$
By the definition of $\I_{k+1}$, we have $L_{I_r}(I_r')\in\I_1$ for all $r\in[m]$.
Thus, $$I_r'=L_{I_r}^{-1}(J_r)\quad\text{for all}\quad 1\leq r\leq m,$$ for some $(J_1,\dots ,J_m)\in\I_1^m$ satisfying $y'\in J_1+\dots +J_m$. But by the $k=1$ case we know there exist at most $g_m$ such tuples $(J_1,\dots,J_m)$. Thus, $$\mathrm{card}\{(I_1,\dots,I_m)\in\I_k^m:y\in I_1+\dots+I_m\}\leq g_m^{k-1}\cdot g_m=g_m^k,$$ which concludes the proof.

\end{proof}
\begin{proof}[Proof of Claim \ref{claim: claim 1}]
Fix $\xi=(\xi_1,\xi_2)\in\R^2$ and $1\leq k\leq K(\delta)+1$. By Claim \ref{claim: claim 2}, there exist at most $g_m^{k-1}$ tuples $(I_1,\dots,I_m)\in\I_{k-1}^m$ satisfying $\xi_1\in I_1+\dots +I_m$. Each $I_r\in\I_{k-1}$ contains $N-1$ intervals of $\I_k'$. Thus, there exist no more than $N^m\cdot g_m^{k-1}$ tuples $(I_1,\dots,I_m)\in(\I_k')^m$ satisfying $\xi_1\in I_1+\dots+I_m$. Claim \ref{claim: claim 1} then follows from the fact $$\xi\in B_{I_1}+\dots+B_{I_m}\implies \xi_1\in I_1+\dots+I_m,$$ since we have seen that $I\subset \text{proj}_1(B_I)$ for all $I\in\I(\delta)$.
\end{proof}

In light of the above, \eqref{eqn: B_delta partition} is a partition of $\mathfrak{B}_\delta$ into subcollections of $\delta$-caps where each subcollection has the finite overlap property as stated in Claim \ref{claim: claim 1}. Therefore, using the definition from \S\ref{sec: background} we get \begin{equation}\label{eqn: Xi_m bound}\Xi_m(\Omega_{2m},\delta)\leq (K(\delta)+1)^{2m}\cdot N^mg_m^{K(\delta)}.\end{equation} Recall that $$\mathcal{E}_{m,\Omega_{2m}}=\limsup_{\delta\to 0^+}\frac{\log\Xi_m(\Omega_{2m},\delta)}{\log\delta^{-1}}.$$ 
Now $$\frac{\log\Xi_m(\Omega_{2m},\delta)}{\log\delta^{-1}}\leq 2m\frac{\log{(K(\delta)+1)}}{\log\delta^{-1}}+m\frac{\log N}{\log\delta^{-1}}+K(\delta)\frac{\log g_m}{\log\delta^{-1}}.$$
Since $N$ is chosen independent of $\delta$, we have $$\frac{\log N}{\log\delta^{-1}}\to 0\quad\text{as}\quad \delta\to 0^+.$$
By \eqref{eqn: K(delta) bound}, we have $$\limsup_{\delta\to 0^+}\frac{\log{(K(\delta)+1)}}{\log\delta^{-1}}=\limsup_{\delta\to 0^+}\frac{\log\log\delta^{-1}}{\log\delta^{-1}}=0.$$
Thus $$\mathcal{E}_{m,\Omega_{2m}}\leq \frac{1}{2m}\frac{\log g_m}{\log N},$$ where we have used \eqref{eqn: K(delta) bound} again. Depending on $m$ and $\epsilon$, let us choose $N$ large enough so that 
\begin{equation}
    \label{eqn: choice of N}
\frac{1}{2m}\frac{\log g_m}{\log N}\leq m\epsilon,
\end{equation}
which yields
\begin{equation}\label{eqn: energy of omega}\mathcal{E}_{m,\Omega_{2m}}\leq m\epsilon,\end{equation} as required. 

\begin{remark}
   If we carried out the same construction with $g=1$ instead, the dimension of $\Omega_{2m}$ would be slightly less than $\frac{1}{2m}$. By taking $g$ to be larger than $1$, we ensure that $\kappa_\Omega=\frac{1}{2m}$ is attainable. The trade off is that larger values of $g$ lead to higher additive energy. Thus, the role of $g$ is to move any dimensional loss to the energy loss. 
\end{remark}
\section{Decoupling estimates}\label{sec: decoupling}
We now look into decoupling estimates, which are crucial for getting improved range of $L^p$ bounds over Theorem \ref{thm: cladek}. For the remainder of this section, $f_J$ will denote an $\mathcal{S}(\R)$ function with Fourier support in the interval $J$; and $f_{\theta_J}$ will denote an $\mathcal{S}(\R^2)$ function with Fourier support in the parallelogram $\theta_J$. Let $\epsilon$ be as defined in the previous section. 
\subsection{One-dimensional decoupling} For a generalized Cantor set $C$, let $\nu_p(C)$ be the quantity defined in the statement of Theorem \ref{thm: chang}. For $p>2$ let $C^p$ be the generalized Cantor set constructed in \S \ref{sec: construction subsection}. In this subsection, we are going to prove the following decoupling estimate.
\begin{prop}[1d decoupling]
    \label{prop: 1d decoupling} 
    There exists a constant $d_p\geq 1$ (independent of $N$) such that for all $k\geq 1$ we have \begin{equation}
         \label{eqn: lambda p decoupling ver 1}
     \|\sum_{I\in\I_k}f_I\|_{L^p(\R)}\leq d_p^k\big(\sum_{I\in\I_k}\|f_I\|_{L^p(\R)}^2\big)^{1/2}.\end{equation}
Consequently, if $N$ is chosen sufficiently large, then $\nu_p(C^p)\leq (p/2)\epsilon$.     
\end{prop}
\begin{proof}
      We prove \eqref{eqn: lambda p decoupling ver 1} by inducting on $k\geq 1$. First we use Lemma \ref{lemma: *} (iii) to obtain the global decoupling estimate $$\|\sum_{I\in\I(N;p)}f_I\|_{L^p(\R)}\leq d_p\big(\sum_{I\in\I(N;p)}\|f_I\|_{L^p(\R)}^2\big)^{1/2}.
$$
    Since $\I_1=\I(N;p)$, this will establish the base case for the induction.

    \medskip\noindent\underline{Base case:}
 Let $\mathcal{Q}$ be a covering of $\R$ by an essentially disjoint family of intervals of length $N^{p/2}$. By Lemma \ref{lemma: *} (iii), for each $Q\in\mathcal{Q}$ we have $$\|\sum_{I\in\I_1}f_I\|_{L^p(Q)}\leq c_p\big(\sum_{I\in\I_1}\|f_I\|^2_{L^p(w_Q)}\big)^{1/2}.$$ 
 Using this we get $$\|\sum_{I\in\I_1}f_I\|_{L^p(\R)}=\bigg(\sum_{Q\in\mathcal{Q}}\|\sum_{I\in\I_1}f_I\|_{L^p(Q)}^p\bigg)^{1/p}\leq c_p\bigg(\sum_{Q\in\mathcal{Q}}\big(\sum_{I\in\I_1}\|f_I\|_{L^p(w_Q)}^2\big)^{p/2}\bigg)^{1/p}.$$ Since $p>2$, by Minkowski's inequality we change the order of summation $$\|\sum_{I\in\I_1}f_I\|_{L^p(\R)}\leq c_p\bigg(\sum_{I\in\I_1}\big(\sum_{Q\in\mathcal{Q}}\|f_I\|_{L^p(w_Q)}^p\big)^{2/p}\bigg)^{1/2}.$$ The inner sum can be bounded as follows $$\sum_{Q\in\mathcal{Q}}\|f_I\|_{L^p(w_Q)}^p=\int_{\R}|f_I|^p\sum_{Q\in\mathcal{Q}}w_Q\lesssim\|f_I\|_{L^p(\R)}^p,$$ by \eqref{eqn: sum of weights is bounded}. Thus, there exists a constant $d_p\geq 1$ such that $$\|\sum_{I\in\I_1}f_I\|_{L^p(\R)}\leq d_p\big(\sum_{I\in\I_1}\|f_I\|_{L^p(\R)}^2\big)^{1/2},$$ establishing the $k=1$ case.

\medskip\noindent\underline{Induction hypothesis:} For the constant $d_p$ as above, assume that for some $k\geq 1$ we have $$\|\sum_{I\in\I_k}f_I\|_{L^p(\R)}\leq d_p^k\big(\sum_{I\in\I_k}\|f_I\|_{L^p(\R)}^2\big)^{1/2}.$$

 \medskip\noindent\underline{Induction step:} Combining the base case with a rescaling argument shows that $$\|f_I\|_{L^p(\R)}\leq d_p\big(\sum_{J\in\I_{k+1}(I)}\|f_J\|_{L^p(\R)}^2\big)^{1/2}\quad\text{for all $I\in\I_k$}.$$ Using this in the induction hypothesis yields $$\|\sum_{J\in\I_{k+1}}f_J\|_{L^p(\R)}\leq d_p^{k+1}\big(\sum_{J\in\I_{k+1}}\|f_J\|_{L^p(\R)}^2\big)^{1/2}.$$ This closes the induction argument and establishes \eqref{eqn: lambda p decoupling ver 1}. 
 
To prove $\nu_p(C^p)\leq (p/2)\epsilon$, we choose $N$ sufficiently large so that $$\frac{\log d_p}{\log N}\leq \frac{p}{2}\epsilon.$$ This concludes the proof. \end{proof}
\subsection{Two-dimensional decoupling}
By Proposition \ref{prop: 1d decoupling}, we have $\nu_p(C^p)\leq (p/2)\epsilon$. Then applying Theorem \ref{thm: chang} gives $D_{3p,k}(C^p)\lesssim_{p,\epsilon}N(k)^{\frac{3}{4}p\epsilon}$, which we record in the following proposition.
\begin{prop}[2d decoupling]
    \label{prop: 2d decoupling}
    For all integers $k\geq 1$ we have \begin{align*}
    \|\sum_{J\in\I_{k}}f_{\theta_J}\|_{L^{3p}(\R^2)}\lesssim_{\epsilon}N(k)^{\frac{3}{4}p\epsilon}\bigg(\sum_{J\in\I_{k}}\|f_{\theta_J}\|_{L^{3p}(\R^2)}^2\bigg)^{1/2}
    .\end{align*}
\end{prop}
Since $N(k)=N^k$, \eqref{eqn: N K(delta)} implies that $$N(k)\leq N^{K(\delta)}\lesssim_N\delta^{-1/p}\quad\text{for all}\quad 1\leq k\leq K(\delta).$$
  Consequently, for all integers $1\leq k\leq K(\delta)$, we have \begin{align}\label{eqn: 2d decoupling}\|\sum_{J\in\I_{k}}f_{\theta_J}\|_{L^{3p}(\R^2)}\lesssim_{\epsilon}\delta^{-\frac{3}{4}\epsilon}\bigg(\sum_{J\in\I_{k}}\|f_{\theta_J}\|_{L^{3p}(\R^2)}^2\bigg)^{1/2}.\end{align} Decoupling estimates of this form have previously been obtained in \cite{ryou} to prove certain restriction estimates.
  
We use \eqref{eqn: 2d decoupling} to prove the following estimate, which will be useful in \S\ref{sec: BR}.
\begin{corollary}
    \label{cor: br decoupling}
    For all $0<\delta<1/2$ we have 
    \begin{align}\label{eqn: BR decoupling}\|\sum_{J\in\I(\delta)}f_{\theta_J}\|_{L^{3p}(\R^2)}\lesssim_{\epsilon}\delta^{-\frac{4\epsilon}{5}}\bigg(\sum_{J\in\I(\delta)}\|f_{\theta_J}\|_{L^{3p}(\R^2)}^2\bigg)^{1/2}.\end{align}
\end{corollary}
\begin{proof}
    The idea is to combine the non-trivial decoupling estimates \eqref{eqn: 2d decoupling}, with trivial decoupling estimates obtained by the Cauchy--Schwarz inequality. Recall the partition of $\I(\delta)$ as \begin{equation*}\I(\delta)=\I_0'\cup\I_1'\cup\I_2'\cup\dots\cup\I_{K(\delta)}',\end{equation*} where $\I_0'=\I_{K(\delta)}$. We first use the Cauchy--Schwarz inequality to decouple the partitioning sets on the right-hand side above as follows $$\|\sum_{J\in\I(\delta)}f_{\theta_J}\|_{L^{3p}(\R^2)}\leq [K(\delta)+1]^{1/2}\bigg(\sum_{k=0}^{K(\delta)}\|\sum_{J\in\I_k'}f_{\theta_J}\|^2_{L^{3p}(\R^2)}\bigg)^{1/2}.$$
    The right-hand term corresponding to $k=0$ can be treated directly using \eqref{eqn: 2d decoupling}. For $k\geq 1$, we use \eqref{eqn: 2d decoupling} to decouple the intervals of $\I_{k-1}$ as $$\|\sum_{J\in\I_k'}f_{\theta_J}\|_{L^{3p}(\R^2)}\lesssim_\epsilon\delta^{-\frac{3\epsilon}{4}}\bigg(\sum_{I\in\I_{k-1}}\|\sum_{J\in\I_k'(I)}f_{\theta_J}\|_{L^{3p}(\R^2)}^2\bigg)^{1/2},$$ where $\I_k'(I)=\{J\in\I_k':J\subset I\}$. Now recall that $\text{card}(\I_k'(I))= N-1$ for each $I\in\I_{k-1}$. Hence, we can locally decouple the intervals of $\I_k'(I)$ as $$\|\sum_{J\in\I_k'(I)}f_{\theta_J}\|_{L^{3p}(\R^2)}\leq N^{1/2}\bigg(\sum_{J\in\I_k'(I)}\|f_{\theta_J}\|_{L^{3p}(\R^2)}^2\bigg)^{1/2},$$ using Cauchy--Schwarz again. Combining everything we get $$\|\sum_{J\in\I(\delta)}f_{\theta_J}\|_{L^{3p}(\R^2)}\lesssim_\epsilon\delta^{-\frac{3\epsilon}{4}}[K(\delta)+1]^{1/2}N^{1/2}\bigg(\sum_{k=0}^{K(\delta)}\sum_{J\in\I_k'}\|f_{\theta_J}\|_{L^{3p}(\R^2)}^2\bigg)^{1/2}.$$ This implies the desired decoupling estimate \eqref{eqn: BR decoupling}, in light of \eqref{eqn: K(delta) bound}.
\end{proof}
\section{Bochner--Riesz bounds}\label{sec: BR}
Recall that $B^\alpha_\Omega$ is given by $$(B^\alpha_\Omega f)\;\widehat{}\;(\xi)=(1-\rho(\xi))_+^\alpha\widehat{f}(\xi),$$ where $\rho$ is the Minkowski functional of $\Omega$. We derive the estimates \eqref{eqn: main a}--\eqref{eqn: lambda p main b} in this section. By a decomposition of the domain into dyadic `annuli', it suffices to prove the following.
\begin{prop}
    \label{prop: main} Let $\Omega_p$ be the domain constructed in \S\ref{sec: construction subsection}, and $\rho$ be its Minkowski functional.
     Let $\beta\in C^2_0(-1/2,1/2)$ satisfy $$|\beta^{(k)}(t)|\leq 1,\qquad\text{for}\qquad k=0,\dots,4.$$ Define the multiplier $$m_{\delta,\alpha}(\xi):=\delta^\alpha\beta\bigg(\frac{\delta^{-1}}{2}\big(1-\rho(\xi)\big)\bigg).$$
If $\mathcal{F}$ denotes the Fourier transform, then \begin{equation}\label{eqn: lp bound br}
    \bigg\|\mathcal{F}^{-1}[m_{\delta,\alpha}\widehat{f}]\bigg\|_{L^{q}(\R^2)}\lesssim_\epsilon\delta^{\alpha-\nu_{q}-\epsilon}\|f\|_{L^{q}(\R^2)},\end{equation} for $$\nu_q=\begin{cases}
        0,\quad& q=4,\\
        \kappa_{\Omega_p}(1/2-1/3p),\quad& q=3p,\\
        \kappa_{\Omega_p},\quad& q=\infty.
    \end{cases}$$
    When $p=2m$ for an integer $m\geq 2$, we have the additional estimate \begin{equation}\label{eqn: l6m bound br}
    \bigg\|\mathcal{F}^{-1}[m_{\delta,\alpha}\widehat{f}]\bigg\|_{L^{2m}(\R^2)}\lesssim_\epsilon\delta^{\alpha-\nu_{2m}-\epsilon}\|f\|_{L^{2m}(\R^2)},\end{equation}
    where $\nu_{2m}=\kappa_{\Omega_{2m}}(1/2-1/m)$.
\end{prop}
Theorem \ref{thm: main} and Theorem \ref{thm: lambda p main} both follow from the above proposition by interpolation.
\begin{remark}
Alternatively, one could obtain the $q=4$ and $q=\infty$ bounds for $B^\alpha_{\Omega_p}$ from Theorem \ref{thm: seeger--ziesler}.
\end{remark}
The $L^4$ and $L^{2m}$ estimates follows from a stronger version of Theorem \ref{thm: bochner riesz bounds} stated below.
\begin{prop}[{\cite[Proposition 4.1]{cladek_BR}}]\label{prop: BR 2n energy est}
Let $\Omega$ be a convex domain and $n\geq 2$ be an integer. If $m_{\delta,\alpha}$ is the multiplier as in the statement of Proposition \ref{prop: main}, then for all $\eta>0$ we have\begin{equation}\label{eqn: br dimension energy bound} \bigg\|\mathcal{F}^{-1}[m_{\delta,\alpha}\widehat{f}]\bigg\|_{L^{2n}(\R^2)}\lesssim_{\eta}\delta^{[\alpha-\frac{\mathcal{E}_{n,\Omega}}{2n}-\kappa_\Omega(\frac{1}{2}-\frac{1}{n})-\eta]}\|f\|_{L^{2n}(\R^2)}.\end{equation}
\end{prop}
In \S\ref{sec: additive energy}, we showed that the domain $\Omega_{2m}$ enjoys the special property $\mathcal{E}_{m,\Omega_{2m}}\leq m\epsilon$. Using this in \eqref{eqn: br dimension energy bound} with $n=m$, and $\eta=\frac{\epsilon}{2}$, we obtain \eqref{eqn: l6m bound br}. 

For $q=4$, we use the following result of Seeger--Ziesler \cite{SZ}.
\begin{prop}[{\cite[Lemma 2.4]{SZ}}]
Every convex domain $\Omega$ satisfies $\mathcal{E}_{2,\Omega}=0$.
\end{prop}
Using $\mathcal{E}_{2,\Omega}=0$ for $n=2$, and $\eta=\epsilon$, we get the $q=4$ case of \eqref{eqn: lp bound br}.
\begin{remark}
    In \cite[Lemma 2.4]{SZ} it is shown that for all convex domains $\Omega$, there exists a near-optimal covering $\mathfrak{B}_\delta$ of $\partial\Omega$ by $\delta$-caps in the sense that $$N(\Omega,\delta)\lesssim\card{\mathfrak B_\delta}\lesssim\log{\delta^{-1}}N(\Omega,\delta),$$ and a partition $$\mathfrak B_\delta=\bigsqcup_{i=1}^{M_0}\mathfrak B_{\delta,i}\quad\text{with}\quad M_0\lesssim\log{\delta^{-1}},$$ such that for all $1\leq i\leq M_0$, each $\xi\in\R^2$ is contained in at most $M_1=2$ sets in the family $\{B_{1,i}+B_{2,i}:B_{1,i},B_{2,i}\in\mathfrak B_{\delta,i}\}$. Indeed this implies that $\mathcal{E}_{2,\Omega}=0$ for all convex domains $\Omega$.
\end{remark}
For $q=\infty$, we have the following result due to Seeger--Ziesler \cite{SZ}.
\begin{prop}[{\cite[(3.32)]{SZ}}]\label{prop: BR infty dimension est}
    Let $\Omega$ be a convex domain, and $m_{\delta,\alpha}$ be the multiplier as in the statement of Proposition \ref{prop: main}. Then for all $\eta>0$ we have 
    \begin{equation}
        \label{eqn: br dimension bound}
        \bigg\|\mathcal{F}^{-1}[m_{\delta,\alpha}\widehat{f}]\bigg\|_{L^{\infty}(\R^2)}\lesssim_{\eta}\delta^{\alpha-\kappa_\Omega-\eta}\|f\|_{L^{\infty}(\R^2)}.
    \end{equation}
\end{prop}
We obtain the $q=\infty$ case of \eqref{eqn: lp bound br} from \eqref{eqn: br dimension bound} with $\eta=\epsilon$.

Therefore, it remains only to prove the $q=3p$ case of \eqref{eqn: lp bound br}. We restate the result below.
\begin{prop}\label{prop: BR 6m bound}
    Let $\Omega_p$ be the convex domain constructed in \S\ref{sec: construction}, and let $m_{\delta,\alpha}$ be the multiplier appearing in the statement of Proposition \ref{prop: main}. Then for all $\epsilon>0$ we have \begin{equation}
        \label{eqn: 6m bound for br}
        \bigg\|\mathcal{F}^{-1}[m_{\delta,\alpha}\widehat{f}]\bigg\|_{L^{3p}(\R^2)}\lesssim_{\epsilon}\delta^{\alpha-\kappa_{\Omega_p}(1/2-1/3p)-\epsilon}\|f\|_{L^{3p}(\R^2)}.
    \end{equation}
\end{prop}
    In order to prove Proposition \ref{prop: BR 6m bound}, we first need to localise the multiplier $m_{\delta,\alpha}$. 
    \subsection{A partition of unity}
    Following an argument of Seeger--Ziesler \cite[p.~664]{SZ}, we  decompose each $I\in\I(\delta)$ into $O(\log\delta^{-1})$ many subintervals, so that consecutive subintervals have comparable lengths. To this end, let $j_I$ be the smallest natural number satisfying $2^{-j_I}|I|< \delta$, and define the points
$$a_{I,j}:=\begin{cases}
   c_{I}, & j=0,\\
   a_{I,j-1}+2^{-j-1}|I|, & j=1,2,\dots,j_I-1,\\
    r_I,& j=j_I,\\
   a_{I,j+1}-2^{-j-1}|I|,& j=-1,,-2,\dots,-j_I-1,\\
   l_I,& j=-j_I.
\end{cases}$$ We define the subintervals $I_{j}:=[a_{I,j-1},a_{I,j}]$ for $|j|=0,1,\dots,j_I$. 
Then $$I=\bigcup_{|j|\leq j_I}I_{j},\quad\text{and}\quad [-1/2,1/2]=\bigcup_{I\in\I(\delta)}\bigcup_{|j|\leq{j_I}}I_{j}.$$
Let $\mathcal{J}(\delta):=\{I_{j}:I\in\I(\delta),|j|\leq j_I\}$ denote the collection of these subintervals. It is clear by the construction, that for consecutive intervals $J_1,J_2\in\J$, we have $$\frac{1}{2}\leq|J_1|\cdot |J_2|^{-1}\leq 2.$$
Let $\beta\in C^\infty_0(-1/2,1/2)$ satisfy $0\leq\beta(t)\leq 1$ for all $t\in[-1/2,1/2]$ and $\beta(t)=1$ for all $t\in [-1/4,1/4]$. If $2J$ denotes the interval concentric to $J$ with twice the length, define $\bar\beta_J(t):=\beta(|2J|^{-1}(t-c_J))$, where $c_J$ denotes the centre of $J\in\J$. Observe that $$1\leq\sum_{J\in\J}\bar\beta_J(t)\leq 4\qquad\text{for all}\qquad t\in [-1/2,1/2].$$ Define $\tilde\beta_J$ as $$\tilde\beta_J:=\frac{\bar\beta_J}{\sum_{I\in\J}\bar\beta_I}.$$ 
Then, by construction $\sum_{J\in\J}\tilde\beta_J=1$, and there exists an absolute constant $c>0$ such that for $k=0,\dots,4$ we have 
\begin{equation*}
      \|\tilde\beta_J^{(k)}\|_{L^\infty(\R)}\leq c\cdot |J|^{-k}\qquad\text{for all}\qquad J\in\J.
\end{equation*}
Finally, define $\beta_J:=c^{-1}\cdot\tilde\beta_J$ for all $J\in\J$.
Define $$\mathcal{B}:=\bigg\{\beta\in C^2_0(-1/2,1/2):\|\beta^{(k)}\|_{L^\infty(\R)}\leq 1\;\text{for}\;k=0,\dots,4\bigg\}.$$ Then each $\beta_J$ is of the form \begin{equation}\label{eqn: beta}\beta_J(t)=\beta(|J|^{-1}(t-c_J))\qquad\text{for some}\qquad \beta\in\mathcal{B}.\end{equation}
 Let us denote $K_{\delta,\alpha}:=\mathcal{F}^{-1}(m_{\delta,\alpha})$, so that $\mathcal{F}^{-1}[m_{\delta,\alpha}\widehat{f}]=K_{\delta,\alpha}*f$.
We can write $$K_{\delta,\alpha}=c^{-1}\sum_{J\in\J}K^J_{\delta,\alpha},\quad\text{where}\quad (K_{\delta,\alpha}^J)\;\widehat{}\;(\xi):=m_{\delta,\alpha}(\xi)\cdot\beta_J(\xi_1).$$
The following result is a consequence of \cite[Proposition 3.2]{SZ}.
\begin{prop}
    \label{prop: kernel estimate}
    For $K_{\delta,\alpha}^J$ as defined above we have $$\|K^J_{\delta,\alpha}\|_{L^1(\R^2)}\lesssim \delta^{\alpha}(\log{\delta^{-1}})\quad\text{uniformly for all}\quad J\in\J.$$
\end{prop}
\begin{proof}
    Let us verify the hypotheses of \cite[Proposition 3.2]{SZ}.
    First, observe that our multiplier has the form $$(K_{\delta,\alpha}^J)\;\widehat{}\;(\xi)=\delta^\alpha\beta_1\bigg(\frac{\delta^{-1}}{2}\big(1-\rho(\xi)\big)\bigg)\cdot\beta_2\big(|J|^{-1}(\xi_1-c_J)\big)\quad\text{for}\quad \beta_1,\beta_2\in\mathcal{B},$$ where we recall $$\mathcal{B}=\bigg\{\beta\in C^2_0(-1/2,1/2):\|\beta^{(k)}\|_{L^\infty(\R)}\leq 1\;\text{for}\;k=0,\dots,4\bigg\}.$$
    Thus, the multiplier $\delta^{-\alpha}(K_{\delta,\alpha}^J)\;\widehat{}\;$ has the same form as \cite[(3.14)]{SZ}.
    
    Next, observe that for all $I\in\I(\delta)$ we have \begin{equation}
        \label{eqn: gamma' bound}
        (t-s)(\gamma_L'(t)-\gamma_R'(s))< 2\delta\quad\text{for all}\quad s,t\in I,\;s<t.
    \end{equation}
    Clearly for each $I\in\Ir$, we have $\gamma'(t)=r_I+l_I$ for all $l_I<t<r_I$, so that \eqref{eqn: gamma' bound} holds. On the other hand, for $I\in\I_{K(\delta)}$ we have $$\gamma_L'(t)-\gamma_R'(s)\leq\gamma_L'(r_I)-\gamma_R'(l_I)\leq 2(r_I-l_I),$$ which implies \eqref{eqn: gamma' bound} since $|I|=r_I-l_I<\delta^{1/2}$.
    Now for each $I\in\I(\delta)$ and $|j|\leq j_I$, notice that one of the following must occur. $$\text{Either}\quad 2I_j\subset I,\quad\text{or}\quad |I_j|\leq 4\delta.$$ Since $\gamma_L'(t)-\gamma_R'(s)\leq 2$ for all $s<t,s,t\in [-1/2,1/2]$, it follows that for all $J\in\J$ we have
\begin{equation}
        \label{eqn: gamma' bound 2}
        (t-s)(\gamma_L'(t)-\gamma_R'(s))< 8\delta\quad\text{for all}\quad s<t, s,t\in 2J.
    \end{equation} 
    Thus, the condition \cite[(3.12)]{SZ} is satisfied. Finally, we note that by construction, each $J\in\J$ satisfies $|J|\geq \delta/2$.
    
    Therefore, we can apply \cite[Proposition 3.2]{SZ} with $2^{-\ell}=\delta$ and $M=0$. The desired bound follows from \cite[(3.16)]{SZ}, applied to the multipliers $\delta^{-\alpha}(K_{\delta,\alpha}^J)\;\widehat{}\;$.
\end{proof}
We are ready to prove the estimate \eqref{eqn: 6m bound for br}.
\subsection{Proof of Proposition \ref{prop: BR 6m bound}}
 By Corollary \ref{cor: br decoupling}, we have the following decoupling estimate \begin{equation}
        \label{eqn: 6m decoupling step 1}
        \|K_{\delta,\alpha}*f\|_{L^{3p}(\R^2)}\lesssim_\epsilon\delta^{-\frac{4\epsilon}{5}}\bigg(\sum_{I\in\I(\delta)}\|K_{\delta,\alpha}^I*f\|^2_{L^{3p}(\R^2)}\bigg)^{1/2},
    \end{equation}
    where $K^I_{\delta,\alpha}:=\sum_{|j|\leq j_I}K^{I_j}_{\delta,\alpha}$.
 Then, for each $I\in\I(\delta)$, we can decouple the kernels $K_{\delta,\alpha}^{I_j}$ using Cauchy--Schwarz as follows 
 \begin{equation*}
 \|K_{\delta,\alpha}^I*f\|_{L^{3p}(\R^2)}\lesssim(\log\delta^{-1})^{1/2}\bigg(\sum_{|j|\leq j_I}\|K_{\delta,\alpha}^{I_j}*f\|^2_{L^{3p}(\R^2)}\bigg)^{1/2}, \end{equation*}
 since $j_I\lesssim\log{\delta^{-1}}$. Combined with \eqref{eqn: 6m decoupling step 1}, this gives \begin{equation*}
     \|K_{\delta,\alpha}*f\|_{L^{3p}(\R^2)}\lesssim_\epsilon\delta^{-\frac{4\epsilon}{5}}(\log\delta^{-1})^{1/2}\bigg(\sum_{J\in\J}\|K_{\delta,\alpha}^{J}*f\|^2_{L^{3p}(\R^2)}\bigg)^{1/2}.
 \end{equation*}
 
 Using H\"older's inequality, and the fact $\text{card}(\I(\delta))\approx \delta^{-\kappa_{\Omega_p}}$, we get $$\bigg(\sum_{J\in\J}\|K_{\delta,\alpha}^{J}*f\|^2_{L^{3p}(\R^2)}\bigg)^{1/2}\lesssim_{\epsilon,\eta}\delta^{-\kappa_{\Omega_p}(\frac{1}{2}-\frac{1}{3p})-\frac{4\epsilon}{5}-\eta}\bigg(\sum_{J\in\J}\|K_{\delta,\alpha}^{J}*f\|^{3p}_{L^{3p}(\R^2)}\bigg)^{1/3p},$$
which, upon choosing $\eta>0$ sufficiently small, gives \begin{equation}
    \label{eqn: 6m decoupling step 2}
 \|K_{\delta,\alpha}*f\|_{L^{3p}(\R^2)}\lesssim_\epsilon\delta^{-\kappa_{\Omega_p}(\frac{1}{2}-\frac{1}{3p})-\frac{5\epsilon}{6}}\bigg(\sum_{J\in\J}\|K_{\delta,\alpha}^{J}*f\|^{3p}_{L^{3p}(\R^2)}\bigg)^{1/3p}.\end{equation} 
In order to estimate the right-hand term on the above display, we use the following lemma.
\begin{lemma}
     \label{lemma: mixed norm interpolation}
     For all $2\leq q\leq\infty$, we have the estimate $$\bigg(\sum_{J\in\J}\|K_{\delta,\alpha}^{J}*f\|^{q}_{L^{q}(\R^2)}\bigg)^{1/q}\lesssim \delta^{\alpha}(\log\delta^{-1})\cdot\|f\|_{L^q(\R^2)}.$$ 
 \end{lemma}
Applying Lemma \ref{lemma: mixed norm interpolation} with $q=3p$ in \eqref{eqn: 6m decoupling step 2} yields  
 \begin{equation}
     \|K_{\delta,\alpha}*f\|_{L^{3p}(\R^2)}\lesssim_\epsilon\delta^{\alpha-\kappa_{\Omega_p}(1/2-1/3p)-\epsilon}\|f\|_{L^{3p}(\R^2)},
 \end{equation}
 which is the desired estimate. Thus, it only remains to prove Lemma \ref{lemma: mixed norm interpolation}.

\begin{proof}[Proof of Lemma \ref{lemma: mixed norm interpolation}]
    We prove this by mixed-norm interpolation with the endpoints $q=2$ and $q=\infty$.  
    
   \medskip\noindent \underline{$q=2:$} By Plancherel's theorem $$\|K_{\delta,\alpha}^J*f\|_{L^2(\R^2)}=\big\|m_{\delta,\alpha}(\xi)\beta_J(\xi_1)\widehat{f}(\xi)\big\|_{L^2(\R^2)}.$$ Since $$m_{\delta,\alpha}(\xi)=\delta^\alpha\beta\bigg(\frac{\delta^{-1}}{2}\big(1-\rho(\xi)\big)\bigg)\quad\text{with}\quad \beta\in\mathcal{B},$$ it follows that
$$\sum_{J\in\J}\|K_{\delta,\alpha}^{J}*f\|^{2}_{L^{2}(\R^2)}\leq\delta^{2\alpha}\sum_{J\in\J}\big\|\beta_J(\xi_1)\widehat{f}(\xi)\big\|_{L^2(\R^2)}^2.$$ By interchanging the order of norms we get $$\sum_{J\in\J}\big\|\beta_J(\xi_1)\widehat{f}(\xi)\big\|_{L^2(\R^2)}^2=\bigg\|\bigg(\sum_{J\in\J}|\beta_J(\xi_1)\widehat{f}(\xi)|^2\bigg)^{1/2}\bigg\|_{L^2(\R^2)}^2.$$ Using the nesting of $\ell^q$ norms and the bound $$\sum_{J\in\J}\beta_J(\xi_1)\lesssim 1,$$ we obtain the desired $\ell^2 L^2$ bound.
   
\medskip\noindent
    \underline{$q=\infty:$} By Young's convolution inequality we can bound the $\ell^\infty L^\infty$ norm as $$\max_{J\in\J}\|K^{J}_{\delta,\alpha}*f\|_{L^\infty(\R^2)}\leq\max_{J\in\J}\|K^{J}_{\delta,\alpha}\|_{L^1(\R^2)}\cdot\|f\|_{L^\infty(\R^2)}.$$
    But by Proposition \ref{prop: kernel estimate} we have $$\max_{J\in\J}\|K^{J}_{\delta,\alpha}\|_{L^1(\R^2)}\lesssim\delta^{\alpha}(\log{\delta^{-1}}),$$ which proves the $q=\infty$ estimate.
\end{proof}
\section{Future directions}
There are a number of relevant questions that could be explored in the future:
Theorem \ref{thm: lambda p main} shows that for each $\kappa\in (0,\frac{1}{2})$ there is a domain $\Omega$ of dimension $\kappa_\Omega=\kappa$ for which Theorem \ref{thm: seeger--ziesler} is not sharp. It is remarked in \cite{SZ} that Theorem \ref{thm: seeger--ziesler} is sharp when $\partial\Omega$ is $C^2$. For $C^2$ domains, Taylor's theorem shows that $\kappa_\Omega=\frac{1}{2}$. Thus Theorem \ref{thm: seeger--ziesler} is sharp for a large class of domains of dimension $\frac{1}{2}$. It is, therefore, natural to ask if Theorem \ref{thm: seeger--ziesler} is sharp for all domains of dimension $\frac{1}{2}$. 

Although Theorem \ref{thm: seeger--ziesler} is clearly sharp for domains of dimension $0$, we can ask a different question. Consider the ball multiplier $B^0_\Omega$ defined as $$(B^0_\Omega f)\;\widehat{}\;:=\chi_\Omega\widehat{f}.$$ When $\Omega$ is the Euclidean ball, it was famously shown by C. Fefferman \cite{fefferman_ball_multiplier} that $B^0_\Omega$ is unbounded on $L^p(\R^2)$ for all $p\neq 2$. Using Bateman's \cite{bateman} characterisation of planar sets of directions which admit Kakeya sets, along with Fefferman's \cite{fefferman_ball_multiplier} argument for the Euclidean ball, Cladek \cite{cladek_BR} showed that $B^0_\Omega$ is unbounded on $L^p(\R^2)$ for $p\neq 2$ whenever $\kappa_\Omega>0$. It would be interesting to characterise all domains satisfying $\kappa_\Omega=0$\footnote{Simple examples show that $\kappa_\Omega=0$ is not sufficient for this.} for which $B^0_\Omega$ is bounded on $L^p(\R^2)$ for $p\neq 2$. This would include polygonal domains (\textit{cf}. \cite{cordoba_polygon}) and also certain domains with countably infinite sides (\textit{cf.} \cite{cordoba--fefferman_dyadic_polygon}). 

As remarked in \cite{cladek_BR}, the bounds obtained in Theorem \ref{thm: main} could be further improved from refined $L^p$ estimates for Nikodym maximal functions with restricted sets of directions.  
\bibliographystyle{plain}
\bibliography{Citation.bib}

\begin{thebibliography}{10}

\bibitem{bateman}
Michael Bateman.
\newblock Kakeya sets and directional maximal operators in the plane.
\newblock {\em Duke Math. J.}, 147(1):55--77, 2009.

\bibitem{bc}
R.~C. Bose and S.~Chowla.
\newblock Theorems in the additive theory of numbers.
\newblock {\em Comment. Math. Helv.}, 37:141--147, 1962/63.

\bibitem{Bourgain_lambdap}
J.~Bourgain.
\newblock Bounded orthogonal systems and the {$\Lambda(p)$}-set problem.
\newblock {\em Acta Math.}, 162(3-4):227--245, 1989.

\bibitem{bourgain--demeter}
Jean Bourgain and Ciprian Demeter.
\newblock The proof of the {$l^2$} decoupling conjecture.
\newblock {\em Ann. of Math. (2)}, 182(1):351--389, 2015.

\bibitem{chang}
Alan Chang, Jaume de~Dios~Pont, Rachel Greenfeld, Asgar Jamneshan, Zane~Kun Li, and Jos\'{e} Madrid.
\newblock Decoupling for fractal subsets of the parabola.
\newblock {\em Math. Z.}, 301(2):1851--1879, 2022.

\bibitem{cladek_BR}
Laura Cladek.
\newblock New ${L}^p$ bounds for {B}ochner-{R}iesz multipliers associated to convex planar domains with rough boundary.
\newblock Preprint: \verb+ arXiv:1509.05106+, 2015.

\bibitem{cordoba_polygon}
A.~C\'{o}rdoba.
\newblock The multiplier problem for the polygon.
\newblock {\em Ann. of Math. (2)}, 105(3):581--588, 1977.

\bibitem{cordoba--fefferman_dyadic_polygon}
A.~Cordoba and R.~Fefferman.
\newblock On the equivalence between the boundedness of certain classes of maximal and multiplier operators in {F}ourier analysis.
\newblock {\em Proc. Nat. Acad. Sci. U.S.A.}, 74(2):423--425, 1977.

\bibitem{Cordoba}
Antonio C\'ordoba.
\newblock A note on {B}ochner-{R}iesz operators.
\newblock {\em Duke Math. J.}, 46(3):505--511, 1979.

\bibitem{fefferman_ball_multiplier}
Charles Fefferman.
\newblock The multiplier problem for the ball.
\newblock {\em Ann. of Math. (2)}, 94:330--336, 1971.

\bibitem{fefferman_biorthogonality}
Charles Fefferman.
\newblock A note on spherical summation multipliers.
\newblock {\em Israel J. Math.}, 15:44--52, 1973.

\bibitem{shortdecoupling}
Shaoming Guo, Zane~Kun Li, Po-Lam Yung, and Pavel Zorin-Kranich.
\newblock A short proof of {$\ell^2$} decoupling for the moment curve.
\newblock {\em Amer. J. Math.}, 143(6):1983--1998, 2021.

\bibitem{laba--wang}
Izabella \L~aba and Hong Wang.
\newblock Decoupling and near-optimal restriction estimates for {C}antor sets.
\newblock {\em Int. Math. Res. Not. IMRN}, (9):2944--2966, 2018.

\bibitem{Lindstrom}
Bernt Lindstr\"{o}m.
\newblock {$B_h[g]$}-sequences from {$B_h$}-sequences.
\newblock {\em Proc. Amer. Math. Soc.}, 128(3):657--659, 2000.

\bibitem{mian--chowla}
Abdul~Majid Mian and S.~Chowla.
\newblock On the {$B_2$} sequences of {S}idon.
\newblock {\em Proc. Nat. Acad. Sci. India Sect. A}, 14:3--4, 1944.

\bibitem{o'bryant}
Kevin O'Bryant.
\newblock A complete annotated bibliography of work related to {S}idon sequences.
\newblock {\em Electron. J. Combin.}, DS11(Dynamic Surveys):39, 2004.

\bibitem{plagne}
Alain Plagne.
\newblock Recent progress on finite {$B_h[g]$} sets.
\newblock In {\em Proceedings of the {T}hirty-second {S}outheastern {I}nternational {C}onference on {C}ombinatorics, {G}raph {T}heory and {C}omputing ({B}aton {R}ouge, {LA}, 2001)}, volume 153, pages 49--64, 2001.

\bibitem{rudin}
Walter Rudin.
\newblock Trigonometric series with gaps.
\newblock {\em J. Math. Mech.}, 9:203--227, 1960.

\bibitem{ryou}
Donggeun Ryou.
\newblock {N}ear{-}optimal restriction estimates for {C}antor sets on the parabola.
\newblock To appear: \textit{IMRN}. arXiv:2301.08651.

\bibitem{Ryou2022}
Donggeun Ryou.
\newblock A variant of the {$\Lambda(p)$}-set problem in {O}rlicz spaces.
\newblock {\em Math. Z.}, 302(4):2545--2566, 2022.

\bibitem{SZ}
Andreas Seeger and Sarah Ziesler.
\newblock Riesz means associated with convex domains in the plane.
\newblock {\em Math. Z.}, 236(4):643--676, 2001.

\bibitem{sidon}
S.~Sidon.
\newblock Ein {S}atz \"{u}ber trigonometrische {P}olynome und seine {A}nwendung in der {T}heorie der {F}ourier-{R}eihen.
\newblock {\em Math. Ann.}, 106(1):536--539, 1932.

\bibitem{Talagrand2014}
Michel Talagrand.
\newblock {\em Upper and lower bounds for stochastic processes}, volume~60 of {\em Ergebnisse der Mathematik und ihrer Grenzgebiete. 3. Folge. A Series of Modern Surveys in Mathematics [Results in Mathematics and Related Areas. 3rd Series. A Series of Modern Surveys in Mathematics]}.
\newblock Springer, Heidelberg, 2014.
\newblock Modern methods and classical problems.

\bibitem{wooley}
Trevor~D. Wooley.
\newblock Nested efficient congruencing and relatives of {V}inogradov's mean value theorem.
\newblock {\em Proc. Lond. Math. Soc. (3)}, 118(4):942--1016, 2019.

\end{thebibliography}
\end{document}